\documentclass[11pt]{amsart}

\numberwithin{equation}{section}
\usepackage{amsmath,amsthm,amsfonts,amscd,eucal}

\newtheorem{theorem}{Theorem}[section]
\newtheorem{corollary}[theorem]{Corollary}
\newtheorem{lemma}[theorem]{Lemma}
\newtheorem{proposition}[theorem]{Proposition}
\theoremstyle{definition}

\newtheorem{remark}[theorem]{Remark}

\numberwithin{equation}{section}

\title[Minimum modulus theorem and ultradifferential operators]{A negative minimum modulus theorem and surjectivity
of ultradifferential operators}

\author[L. Zsid\'o]{L\'aszl\'o Zsid\'o}
\address[L. Zsid\'o]{Dipartimento di Matematica, Universit\`a di Roma "Tor Vergata",
Via della Ricerca Scientifica 1, 00133 Roma, ITALIA}
\email{\tt zsido@mat.uniroma2.it}

\dedicatory{Dedicated to the memory of Professor Ciprian Foia\c s}



\thanks{Supported by INdAM and EU}

\date{15 September 2020}

\begin{document}

\begin{abstract}
In 1979 I. Cior\u{a}nescu and L. Zsid\'o have proved a minimum modulus theorem
for entire functions dominated by the restriction to $(0\;\! ,+\infty )$ of entire functions
of the form $\displaystyle \omega (z)=\prod\limits_{j=1}^{\infty}\Big( 1+\frac{iz}{t_j}\Big) ,
z\in\mathbb{C}\;\!$, with $\displaystyle 0<t_1\leq t_2\leq t_3\leq\;\! ...\;\! \leq +\infty\;\! ,\;\!
t_1<+\infty\;\! ,\;\! \sum\limits_{j=1}^{\infty}\frac 1{t_j}<+\infty\;\!$, and such

\noindent that $\!\displaystyle \int\limits_1^{+\infty}\frac{\ln |\omega (t)|}{t^2} \ln
\frac t{\ln |\omega (t)|}\;\!{\rm d}t<+\infty\;\!$. It implies that for $\omega$ as above,
every $\omega$-ultradifferential operator with constant coefficients and of convergence
type maps some $\mathcal{D}_\rho{\!\! '}\supset\mathcal{D}_\omega{\!\! '}$ onto
itself. Here we show that the above results are sharp, by proving the negative counterpart
of the above minimum modulus theorem$\;\! :$ if $\!\displaystyle \int\limits_1^{+\infty}
\frac{\ln |\omega (t)|}{t^2} \ln\frac t{\ln |\omega (t)|}\;\!{\rm d}t=+\infty\;\!$, then always
there exists an entire function dominated by the restriction to $(0\;\! ,+\infty )$ of $\omega\;\!$,
which does not satisfy the minimum modulus conclusion in the 1979 paper.
It follows that for such $\omega$ there exists an $\omega$-ultradifferential operator with 
constant coefficients and of convergence type, which does not map any $\mathcal{D}_\rho{\!\! '}
\supset\mathcal{D}_\omega{\!\! '}$ onto itself.
\end{abstract}

\maketitle


\section{Introduction}

The main purpose of this paper is to expose (in a slightly completed

\noindent form) the surjectivity criterion for ultradifferential operators with constant coefficients,
given in \cite{C-Z2}, Proposition 2.7, and to prove that this criterion is sharp.

To avoid ambiguity, we notice that
we will use Bourbaki's terminology: "positive" and "strictly positive" instead of "non-negative" and
"positive", as well as "increasing" and "strictly increasing" instead of "non-decreasing" and "increasing".

In Section 2 we present, following \cite{C-Z1}, the current ultradistribution theories on
$\mathbb{R}\;\!$. Up to equivalence, there are two of them.

The first one is parametrized by entire functions of the form
\medskip

\centerline{$\qquad\displaystyle \omega (z)=\prod\limits_{j=1}^{\infty}\Big( 1+\frac{iz}{t_j}\Big)\;\! ,
\qquad z\in\mathbb{C}\;\! ,$}

\centerline{$\qquad\displaystyle \text{where }0<t_1\leq t_2\leq t_3\leq\;\! ...\;\! \leq +\infty\;\! ,\quad
t_1<+\infty\;\! ,\quad \sum\limits_{j=1}^{\infty}\frac 1{t_j}<+\infty\;\! ,$}
\smallskip

\noindent whose set is denoted by $\mathbf{\Omega}\;\!$. $\mathcal{D}_\omega$ is a
strict inductive limit of a sequence of nuclear Fr\'echet spaces, whose elemts are infinitely
differentiable functions of compact support. The strong dual $\mathcal{D}_\omega{\!\! '}$ is
the space of $\omega$-ultradistributions.
$\mathcal{D}_\omega$ can be naturally considered a subspace of $\mathcal{D}_\omega{\!\! '}\;\!$.
If $\omega\;\! ,\rho\in\mathbf{\Omega}$ are such that $|\omega (t)|\leq c\;\! |\rho (t)|$ for some
constant $c>0$ and all $t\in\mathbb{R}\;\!$,
then $\mathcal{D}_\rho\subset\mathcal{D}_\omega$ and $\mathcal{D}_\omega{\!\! '}\subset
\mathcal{D}_\rho{\!\! '}\;\!$.

A second ultradistribution theory is obtained by considering the spaces
$\mathcal{D}_\omega$ and $\mathcal{D}_\omega{\!\! '}$ only for entire functions $\omega$
as above with the $t_j$'s satisfying

\noindent additionally
\smallskip

\centerline{$\qquad\displaystyle 0<t_1\leq \frac{t_2}2\leq \frac{t_3}3\leq\;\! ...\;\! .$}
\medskip

\noindent $\mathbf{\Omega}_0$ will denote the set of these entire functions.

In Section 3 we discuss ultradifferential operators and formulate the main results.

We call a linear map $T : \mathcal{D}_\omega\longrightarrow \mathcal{D}_\omega$
$\omega$-ultradifferential operator whenever the support of $T\varphi$ is contained in
the support of $\varphi\in\mathcal{D}_\omega\;\!$. It is of constant coefficients if it
commutes with the translation operators.

$T$ is an $\omega$-ultradifferential operator of constant coefficients if and only if
there exists an entire function $f$ of exponential type $0$ such that
$|f(it)|\leq c\;\! |\omega (t)|^n ,t\in\mathbb{R}\;\!$, for some $c>0$ and integer
$n\geq 1\;\!$, such that the Fourier transform of $T\varphi$ is the product of the
Fourier transform of $\varphi$ multiplied by $\mathbb{R}\ni t\longmapsto f(it)\;\!$.
In order that $T$ be the convergent Taylor series $f(D)$ of the derivation operator
$D$, $f$ must satisfy the stronger majorization property $|f(z)|\leq c\;\! \big|\omega (|z|)\big|^n,
z\in\mathbb{C}\;\!$, with $c>0$ a constant and $n\geq 1$ an integer. In this case
$T$ is called of convergence type.

Any $\omega$-ultradifferential operator $T$ of constant coefficients can be uniquely
extended to a continuous linear operator $\mathcal{D}_\omega{\!\! '}\longrightarrow
\mathcal{D}_\omega{\!\! '}\;\!$, still denoted by $T$.
A central issue is the characterization of the situation $T\mathcal{D}_\omega{\!\! '}=
\mathcal{D}_\omega{\!\! '}\;\!$, when the equation $f(D)X=F$ has a solution
$X\in \mathcal{D}_\omega{\!\! '}$ for each $F\in \mathcal{D}_\omega{\!\! '}\;\!$,
in terms of the entire function $f$ associated to $T$. Such a criterion was obtained
by I. Cior\u{a}nescu in \cite{Ci}, Proposition 2.4 and Theorem 3.4$\;\! :$
$T\mathcal{D}_\omega{\!\! '}=\mathcal{D}_\omega{\!\! '}$ if and only if $f$ satisfies
a certain minimum modulus condition.

In \cite{C-Z2} a minimum modulus theorem was obtained, which implies that if
\smallskip

\centerline{$\qquad\displaystyle
\int\limits_1^{+\infty}\frac{\ln |\omega (t)|}{t^2}\;\! \ln \frac t{\ln |\omega (t)|}\;\!{\rm d}t
<+\infty$}
\smallskip

\noindent then, for every $\omega$-ultradifferential operator $T$ of constant coefficients
and of convergent type, there exists some $\rho\in\mathbf{\Omega}\;\!$,
$|\omega (t)|\leq c\;\! |\rho (t)|$ for some constant $c>0$ and all $t\in\mathbf{R}\;\!$,
hence such that $\mathcal{D}_\omega{\!\! '}\subset\mathcal{D}_\rho{\!\! '}\;\!$, for
which the surjectivity $T\mathcal{D}_\rho{\!\! '}=\mathcal{D}_\rho{\!\! '}$ holds true.
We complete this result by proving that if $\omega\in\mathbf{\Omega}_0\;\!$, then
we can choose $\rho\in\mathbf{\Omega}_0$ (Theorem \ref{surjectivity}).
To do this, we completed the

\noindent minimum modulus theorem from \cite{C-Z2} correspondingly (Theorem \ref{min}).

On the other hand we prove (Theorem \ref{no-surjectivity}) that if
\smallskip

\centerline{$\qquad\displaystyle
\int\limits_1^{+\infty}\frac{\ln |\omega (t)|}{t^2}\;\! \ln \frac t{\ln |\omega (t)|}\;\!{\rm d}t
=+\infty$}
\smallskip

\noindent then, there exists an $\omega$-ultradifferential operator $T$ of constant coefficients
and of convergent type, such that the surjectivity $T\mathcal{D}_\rho{\!\! '}=\mathcal{D}_\rho{\!\! '}$
can not hold for any $\rho\in\mathbf{\Omega}\;\!$, $|\omega (t)|\leq c\;\! |\rho (t)|$ for some constant
$c>0$ and all $t\in\mathbf{R}$ (Theorem \ref{no-surjectivity}).
This is consequence of the negative minimum modulus theorem (Theorem \ref{no-min}), claiming
that for $\omega$ as above there exists an entire function $f$ such that
$|f(z)|\leq \big|\omega (|z|)\big|^2 , z\in\mathbb{C}\;\!$, but for no increasing function
$\beta : (0\;\! ,+\infty )\longrightarrow (0\;\! ,+\infty )$ with
$\displaystyle \int\limits_1^{+\infty}\frac{\beta (t)}{t^2}\;\!{\rm d}t<+\infty$ can hold the minimum
modulus condition
\begin{equation*}
\sup\limits_{\substack{s\in\mathbb{R} \\ |s-t|\leq \beta (t)}} \ln |f(s)|\geq -\;\! \beta (t)\;\! ,\qquad
t>0\;\! .
\end{equation*}
This negative minimum modulus theorem is the hearth of the paper and is proved in the last,
6th section.

In Section 4 we investigate the majorization of positive functions defined on $(0\;\! ,+\infty )$
with functions $\alpha : (0\;\! ,+\infty )\longrightarrow (0\;\! ,+\infty )$ belonging to different
regularity classes and satisfying the non-quasianalyticity condition
\smallskip

\centerline{$\qquad\displaystyle \int\limits_1^{+\infty}\!\frac{\alpha (t)}{t^2}\;\!{\rm d}t<+\infty\;\! .$}
\smallskip

\noindent (like $(0\;\! ,+\infty )\ni t\longmapsto\ln |\omega (t)|$ for $\omega\in\mathbf{\Omega}$).
These topics are used in the proof of Theorem \ref{min}. Lemma \ref{concave-explicite} could
be of interest for itself.

Section 5 is devoted to increasing functions $\alpha : (0\;\! ,+\infty )\longrightarrow (0\;\! ,+\infty )$

\noindent satisfying
\smallskip

\centerline{$\displaystyle \int\limits_1^{+\infty}\!\frac{\alpha (t)}{t^2}\;\!{\rm d}t<+\infty\;$ and
$\displaystyle \int\limits_1^{+\infty}\!\frac{\alpha (t)}{t^2}\ln \frac t{\alpha (t)}\;\!{\rm d}t=+\infty\;\! .$}
\smallskip

\noindent Discretization of the above conditions is investigated (Propositions \ref{msnq-discrete}
and \ref{msnq-permanence2}) and the case $\alpha (t)=\ln |\omega (t)|\;\! ,\omega\in\mathbf{\Omega}\;\!$,
is characterized (Theorem \ref{msnq-omega}).

\noindent In particular, for $0<t_1\leq t_2\leq t_3\leq\;\! ...\;\! \leq +\infty\;\! , t_1<+\infty\;\!$,
\begin{equation*}
\alpha : (0\;\! ,+\infty )\ni t\longmapsto \ln\Big| \prod\limits_{j=1}^{\infty}\Big( 1+\frac{it}{t_j}\Big)\Big|
\end{equation*}
satisfies the above two conditions if and only if
\medskip

\centerline{$\displaystyle \sum\limits_{j=1}^{\infty}\frac 1{t_j}<+\infty\;$ and
$\displaystyle \sum\limits_{j=1}^{\infty}\frac{\;\!\displaystyle \ln \frac{t_j}j}{t_j}=+\infty\;\! ,$}
\medskip

\noindent what happens, for example, if $t_j=j\;\! (\ln j)\;\! (\ln\ln j)^p ,j\geq 3\;\!$, with $1<p\leq 2\;\!$.
Section 5 actually prepares Section 6.

Finally, in Section 6 the negative minimum modulus theorem Theorem \ref{no-min} is
proved. The proof uses the machinery developed in Section 5 and the key ingredient
Lemma \ref{main-lemma}.
I am indebted to Professor W. K. Hayman for the proof of a statement
very close to Lemma \ref{main-lemma} in the case of $n_j=\alpha (2^j)\;\! ,j\geq 2\;\!$, where
$\displaystyle \alpha (t)=\frac t{(\ln t)(\ln\ln t)^2}\;\! ,t>e\;\!$, sent to me in \cite{H}.
The proof of Lemma \ref{main-lemma} is based on Hayman's ideas, it is actually an adaptation
of Hayman's draft to the general case.

\section{Ultradistribution theories}

In order to enlarge the family of L. Schwartz's distributions, I. M. Gelfand and G. E. Shilov
proposed in \cite{G-S1} (see also \cite{G-S2}, Chapters II and IV) the

\noindent following extension of L. Schwartz's strategy:
consider an appropriate locally convex topological vector space $\mathcal{B}$ of infinitely
differentiable functions such that
\begin{itemize}
\item $\mathcal{B}$ is a Fr\'echet space or a countable inductive limit of Fr\'echet spaces,
\item the topology of $\mathcal{B}$ is stronger than the topology of pointwise convergence.
\end{itemize}
The elements of $\mathcal{B}$ are called {\it basic functions}, and the elements of the dual
$\mathcal{B}'$, {\it generalized functions}. If we "shrink" $\mathcal{B}$, then $\mathcal{B}'$
becomes larger.

The generalized functions $\mathcal{B}'$ are usually called {\it ultradistributions} when, roughly
speaking, disjoint compact sets can be separated by functions which belong to $\mathcal{B}$.
This yields a "lower bound" for $\mathcal{B}$. Ultradistribuion theories are mostly based on
non-quasianaliticity.

Let us briefly sketch, following \cite{C-Z1}, Section 7, what we will here understand by an
ultradistribution theory on the real line $\mathbb{R}$ (a slightly different picture is given in
\cite{Sch}).

Let $\mathfrak{S}$ be a parameter set and assume that to each $\sigma\in\mathfrak{S}$
is associated a locally convex topological vector space $\mathcal{D}_\sigma$ of infinitely
differentiable functions $\mathbb{R}\longrightarrow\mathbb{C}$ with compact support such
that, for every $\sigma\in\mathfrak{S}\;\!$,
\begin{itemize}
\item[(i)] $\mathcal{D}_\sigma$ is an inductive limit of a sequence of Fr\'echet spaces;
\item[(ii)] the topology of $\mathcal{D}_\sigma$ is stronger than the topology of pointwise
convergence;
\item[(iii)] $\mathcal{D}_\sigma$ is an algebra under pointwise multiplication;
\item[(iv)] for $K\subset D\subset\mathbb{R}\;\!$, $K$ compact and $D$ open, there exists
$\varphi\in\mathcal{D}_\sigma$ such that
\begin{equation*}
0\leq\varphi\leq 1\;\! ,\quad \varphi (s)=1\text{ for }s\in K\;\! ,\quad {\rm supp} (\varphi )\subset D\;\! ;
\end{equation*}
\item[(v)] denoting by $\mathcal{E}_\sigma$ the multiplier algebra of $\mathcal{D}_\sigma\;\!$,
that is the set of all functions $\psi : \mathbb{R}\longrightarrow\mathbb{C}$ satisfying
$\varphi \psi\in \mathcal{D}_\sigma\;\! ,\varphi\in \mathcal{D}_\sigma\;\!$, and endowing it with
the projective limit topology defined by the linear mappings
\smallskip

\centerline{$\qquad\mathcal{E}_\sigma\ni\psi\longmapsto \varphi \psi\in \mathcal{D}_\sigma\;\! ,\quad
\varphi\in \mathcal{D}_\sigma\;\! ,$}
\smallskip

\noindent the set $\mathcal{A}$ of all real analytic complex functions on $\mathbb{R}$ is a dense
subset of $\mathcal{E}_\sigma\;\!$.
\end{itemize}
We will say that $\{ \mathcal{D}_\sigma\}_{\sigma\in \mathfrak{S}}$ is a {\it theory of ultradistributions}
and the elements of the dual $\mathcal{D}_\sigma{\!\! '}$ will be called {\it $\sigma$-ultradistributions.}

For $\sigma\in\mathfrak{S}$ and $F\in\mathcal{D}_\sigma{\!\! '}\;\!$, there is a smallest closed
set $S\subset\mathbb{R}$ such that
\smallskip

\centerline{$\varphi\in \mathcal{D}_\sigma\;\! , S\cap {\rm supp} (\varphi )=\emptyset
\;\Longrightarrow F(\varphi )=0\;\! .$}
\smallskip

\noindent Then $S$ is called the {\it support} of $F$ and is denoted by ${\rm supp}(F)\;\!$.
The dual $\mathcal{E}_\sigma{\!\! '}$

\noindent can be identified with the vector space of all $\sigma$-ultradistributions of compact
support, since the restriction map $\mathcal{E}_\sigma{\!\! '}\ni G\longmapsto G\lceil \mathcal{D}_\sigma$
is a linear isomorphism of $\mathcal{D}_\sigma{\!\! '}$ onto
$\{ F\in \mathcal{D}_\sigma{\!\! '}\: ;\, {\rm supp}(F)\text{ compact}\}\;\!$.

By a $\sigma$-ultradifferential operator we mean a linear operator $T : \mathcal{D}_\sigma
\longrightarrow \mathcal{D}_\sigma$ which doesn't enlarge the support:
\smallskip

\centerline{${\rm supp}(T\varphi )\subset {\rm supp}(\varphi )\;\! ,\qquad \varphi\in\mathcal{D}_\sigma\;\! .$}
\smallskip

Let $\{ \mathcal{D}_\sigma\}_{\sigma\in \mathfrak{S}}$ and $\{ \mathcal{D}_\tau\}_{\tau\in \mathfrak{T}}$
be two ultradistribution theories. We say that the ultradistribution theory
$\{ \mathcal{D}_\tau\}_{\tau\in \mathfrak{T}}$ is {\it larger} than $\{ \mathcal{D}_\sigma\}_{\sigma\in \mathfrak{S}}$
if for every $\sigma\in \mathfrak{S}$ there exists some $\tau\in \mathfrak{T}$ such that
$\mathcal{D}_\tau\subset \mathcal{D}_\sigma\;\!$, or equivalently, $\mathcal{E}_\tau\subset \mathcal{E}_\sigma\;\!$.
When this happens then the inclusion maps $\mathcal{D}_\tau\hookrightarrow \mathcal{D}_\sigma$ and
$\mathcal{E}_\tau\hookrightarrow \mathcal{E}_\sigma$ are continuous and have a dense range.

We notice that if $\{ \mathcal{D}_\sigma\}_{\sigma\in \mathfrak{S}}$ and
$\{ \mathcal{D}_\tau\}_{\tau\in \mathfrak{T}}$ are ultradistribution theories and
$\{ \mathcal{D}_\tau\}_{\tau\in \mathfrak{T}}$ is larger than $\{ \mathcal{D}_\sigma\}_{\sigma\in \mathfrak{S}}\;\!$,
then
\medskip

\centerline{$\displaystyle \mathcal{A}\subset\bigcap\limits_{\tau\in \mathfrak{T}}\mathcal{E}_\tau\subset
\bigcap\limits_{\sigma\in \mathfrak{S}}\mathcal{E}_\sigma\;\! .$}
\smallskip

We say that two ultradistribution theories $\{ \mathcal{D}_\sigma\}_{\sigma\in \mathfrak{S}}$ and
$\{ \mathcal{D}_\tau\}_{\tau\in \mathfrak{T}}$ are {\it equivalent} whenever each one of them is
larger than the other.

Let us recall the usual ultradistribution theories. They are labeled by one of the following parameter
sets $\mathfrak{S}\;\!$:
\begin{itemize}
\item $\boldsymbol{\mathcal{M}}$ is the set of all sequences $(M_p)_{p\geq 0}$ in $(0,+\infty )\;\!$,
$M_0=1\;\!$, satisfying
\begin{equation*}
\begin{split}
\qquad& M_p^2\leq M_{p-1}M_{p+1}\;\! ,\; p\geq 1\; (\text{logarithmic convexity})\;\! , \\
\qquad& \sum\limits_{p\geq 1}\frac{M_{p-1}}{M_p}<+\infty\; (\text{non-quasianalyticity})\;\! .
\end{split}
\end{equation*}
\item $\boldsymbol{\mathcal{M}}_0$ is the set of all sequences $(M_p)_p\in \boldsymbol{\mathcal{M}}$
which satisfy the stronger  logarithmic convexity condition
\begin{equation*}
\Big(\frac{M_p}{p!}\Big)^2\leq \frac{M_{p-1}}{(p-1)!}\cdot \frac{M_{p+1}}{(p+1)!}\;\! ,\; p\geq 1\;\! .
\end{equation*}
\item $\boldsymbol{\mathcal{A}}$ is the set of all continuous functions $\alpha : \mathbb{R}\longrightarrow
(0,+\infty )$ satisfying
\begin{equation*}
\begin{split}
\qquad \alpha (0)=0\;\! ,\quad \alpha (&t+s)\leq\alpha (t)+\alpha (s)\text{ for }t\;\! ,s\in
\mathbb{R}\;\text{(subadditivity)}\;\! , \\
\qquad \text{there exist }a\in\mathbb{R}&\text{ and }b>0\text{ such that }\alpha (t)\geq a+b \ln (1+|t|)\;\! ,
t\in\mathbb{R}\;\! , \\
&\qquad \int\limits_{-\infty}^{+\infty}\frac{\alpha (t)}{1+t^2}\;\!{\rm d}t<+\infty\;\! .
\end{split}
\end{equation*}
\item $\mathbf{\Omega}$ is the set of all entire functions $\omega$ of the form
\begin{equation}\label{omega}
\qquad \omega (z)=\prod\limits_{j=1}^{\infty}\Big( 1+\frac{iz}{t_j}\Big)\;\! ,\qquad z\in\mathbb{C}\;\! ,
\end{equation}
\begin{equation*}
\qquad \text{where }0<t_1\leq t_2\leq t_3\leq\;\! ...\;\! \leq +\infty\;\! ,\quad t_1<+\infty\;\! ,\quad
\sum\limits_{j=1}^{\infty}\frac 1{t_j}<+\infty\;\! .
\end{equation*}
\item $\mathbf{\Omega}_0$ is the set of all entire functions $\omega$ of the form
\begin{equation*}
\begin{split}
\qquad &\omega (z)=\prod\limits_{j=1}^{\infty}\Big( 1+\frac{iz}{t_j}\Big)\;\! ,\qquad z\in\mathbb{C}\;\! , \\
\qquad \text{where }0<t_1\leq \frac{t_2}2&\leq \frac{t_3}3\leq\;\! ...\;\! \leq +\infty\;\! ,\quad t_1<+\infty\;\! ,
\quad \sum\limits_{j=1}^{\infty}\frac 1{t_j}<+\infty\;\! .
\end{split}
\end{equation*}
\end{itemize}

Let $(M_p)_p\in \boldsymbol{\mathcal{M}}$ be fixed. For $K\subset\mathbb{R}$ compact and
$h>0\;\!$, let $\mathcal{D}_{\{ M_p\}\;\! ,h}(K)$ denote the vector space of all infinitely differentiable
functions $\varphi : \mathbb{R}\longrightarrow\mathbb{C}$ with ${\rm supp}(\varphi )\subset K$,
satisfying
\begin{equation*}
\|\varphi\|_{\{ M_p\}\;\! ,h}:=\sup_{s\in K,\;\! p\geq 0} \frac 1{h^pM_p}\;\! |\varphi^{(p)}(s)|<+\infty\;\! .
\end{equation*}
Then $\mathcal{D}_{\{ M_p\}\;\! ,h}(K)\;\!$, endowed with the norm $\|\;\!\cdot\;\!\|_{\{ M_p\}\;\! ,h}\;\!$,
becomes a Banach space.

The {\it Roumieu ultradifferentiable functions of class} $(M_p)_p\in \boldsymbol{\mathcal{M}}$ on
$\mathbb{R}\;\!$,

\noindent having compact support, are
\begin{equation*}
\mathcal{D}_{\{ M_p\}}:=\lim_{\substack{\longrightarrow \\ K\subset\mathbb{R}\text{ compact}}}\;
\lim_{\substack{\longrightarrow \\ 0<h\to\infty}} \mathcal{D}_{\{ M_p\}\;\! ,h}(K)
\end{equation*}
(see \cite{R1} or \cite{K}), while the {\it Beurling-Komatsu ultradifferentiable functions of class}
$(M_p)_p\in \boldsymbol{\mathcal{M}}$ on $\mathbb{R}\;\!$, having compact support, are
\begin{equation*}
\mathcal{D}_{( M_p)}:=\lim_{\substack{\longrightarrow \\ K\subset\mathbb{R}\text{ compact}}}\;
\lim_{\substack{\longleftarrow \\ 0<h\to 0}} \mathcal{D}_{\{ M_p\}\;\! ,h}(K)
\end{equation*}
(see \cite{K}). $\{ \mathcal{D}_{\{ M_p\}}\}_{(M_p)_p\in \boldsymbol{\mathcal{M}}}$ and
$\{ \mathcal{D}_{( M_p)}\}_{(M_p)_p\in \boldsymbol{\mathcal{M}}}$ are the {\it Roumieu}
resp. {\it Beurling-Komatsu ultradistribution theories}.

Let now $\alpha\in\boldsymbol{\mathcal{A}}$ be fixed. For $K\subset\mathbb{R}$ compact we
denote by $\mathcal{D}_{\alpha}(K)$ the vector space of all continuous functions
$\varphi : \mathbb{R}\longrightarrow\mathbb{C}$ with ${\rm supp}(\varphi )\subset K$, for which
\begin{equation*}
\|\varphi\|_{\alpha ,\lambda}:=\int\limits_{-\infty}^{+\infty}|\widehat{\varphi}(t)|\;\! e^{\lambda\;\! \alpha (t)}
{\rm d}t<+\infty\;\! ,\qquad \lambda >0\;\! ,
\end{equation*}
where $\widehat{\varphi}$ stands for the Fourier transform of $\varphi\;\!$:
\begin{equation*}
\widehat{\varphi}(t)=\frac 1{2\;\!\pi}\int\limits_{-\infty}^{+\infty}\varphi (s)\;\! e^{-its}{\rm d}s
\end{equation*}
Then $\mathcal{D}_{\alpha}(K)\;\!$, endowed with the family of norms $\|\;\!\cdot\;\!\|_{\alpha ,\lambda}\;\!$,
$\lambda >0\;\!$, becomes a Fr\'echet space.

The {\it Beurling-Bj\"orck ultradifferentiable functions of class} $\alpha\in\boldsymbol{\mathcal{A}}$ on
$\mathbb{R}\;\!$,

\noindent having compact support, are
\begin{equation*}
\mathcal{D}_{\alpha}:=\lim_{\substack{\longrightarrow \\ K\subset\mathbb{R}\text{ compact}}}
\mathcal{D}_{\alpha}(K)
\end{equation*}
(see \cite{Be1} and \cite{Bj}). $\{ \mathcal{D}_{\alpha}\}_{\alpha\in\boldsymbol{\mathcal{A}}}$
is the {\it Beurling-Bj\"orck ultradistribution theory}.

Finally, for $\omega\in\mathbf{\Omega}$ and $K\subset\mathbb{R}$ compact, let
$\mathcal{D}_{\omega}(K)$ be the vector space of all continuous functions
$\varphi : \mathbb{R}\longrightarrow\mathbb{C}$ with ${\rm supp}(\varphi )\subset K$, for which
\begin{equation*}
p_{\omega ,n}(\varphi ):=\sup\limits_{t\in\mathbb{R}}|\widehat{\varphi}(t)\;\!\omega (t)^n|<+\infty\;\! ,
\qquad n\geq 1\;\! .
\end{equation*}
Then $\mathcal{D}_{\omega}(K)\;\!$, endowed with the family of norms $p_{\omega ,n}\;\! ,n\geq 1\;\!$,
becomes a Fr\'echet space.

The {\it $\omega$-ultradifferentiable functions} on $\mathbb{R}\;\!$, having compact support, are
\begin{equation*}
\mathcal{D}_{\omega}:=\lim_{\substack{\longrightarrow \\ K\subset\mathbb{R}\text{ compact}}}
\mathcal{D}_{\omega}(K)
\end{equation*}
(see \cite{C-Z1}, Section 2). $\{ \mathcal{D}_{\omega}\}_{\omega\in\mathbf{\Omega}}$ is the
{\it $\omega$-ultradistribution theory}.

We have to remark that in \cite{C-Z1}, Definition III, $\mathcal{D}_{\omega}(K)$ is defined by using
the norms $p_{\omega ,L,n}\;\! ,L>0\;\! ,n\geq 1\;\!$, where
\begin{equation*}
p_{\omega ,L,n}(\varphi ):=\sup\limits_{t\in\mathbb{R}}|\widehat{\varphi}(t)\;\!\omega (L\;\! t)^n|\;\! .
\end{equation*}
However, with the notation of (\ref{omega}), we have
\begin{equation}\label{seminorms}
|\;\!\omega (L\;\! t)|=\prod\limits_{j=1}^{\infty}\Big( 1+\frac{L^2t^2}{t_k}\Big)^{\! 1/2}\leq
\prod\limits_{j=1}^{\infty}\Big( 1+\frac{t^2}{t_k}\Big)^{\! L^2/2}=|\;\!\omega (t)|^{L^2},
\end{equation}
so the two definitions are equivalent.

We notice that the Roumieu, the Beurling-Komatsu and the Beurling-Bj\"orck ultradistribution theories
were considered also on open subsets of $\mathbb{R}^d$ (see \cite{R2}, \cite{K}, \cite{Bj}), while the
$\omega$-ultradistribution theory, originally considered in \cite{C-Z1} only on $\mathbb{R}$, was
subsequently extended to the multidimensional setting (see \cite{BMT} and \cite{Ab}).
However, in this paper we will restrict us to the one-dimensional case of $\mathbb{R}\;\!$.

In \cite{C-Z1}, 7.4 it was shown that the ultradistribution theories
\begin{equation}\label{full}
\{ \mathcal{D}_{\{ M_p\}}\}_{(M_p)_p\in \boldsymbol{\mathcal{M}}}\;\! ,\quad
\{ \mathcal{D}_{( M_p)}\}_{(M_p)_p\in \boldsymbol{\mathcal{M}}}\;\! ,\quad
\{ \mathcal{D}_{\omega}\}_{\omega\in\mathbf{\Omega}}
\end{equation}
are equivalent. Thus they are just different labelings of the same global set of ultradistributions.
To work with ultradifferential operators, the setting of the $\omega$-ultradistribution theory seems
to be the most advantageous. Therefore we will adopt this setting in the sequel.

We notice that, according to \cite{C-Z3}, Theorem 1, also the ultradistribution theories
\begin{equation}\label{restr}
\{ \mathcal{D}_{\{ M_p\}}\}_{(M_p)_p\in \boldsymbol{\mathcal{M}_0}}\;\! ,\quad
\{ \mathcal{D}_{( M_p)}\}_{(M_p)_p\in \boldsymbol{\mathcal{M}_0}}\;\! ,\quad
\{ \mathcal{D}_{\alpha}\}_{\alpha\in\boldsymbol{\mathcal{A}}}\;\! ,\quad
\{ \mathcal{D}_{\omega}\}_{\omega\in\mathbf{\Omega}_0}
\end{equation}
are equivalent. As was pointed out in \cite{C-Z1}, Section 7.7, $\bigcap\limits_{\omega\in\mathbf{\Omega}}
\mathcal{E}_{\omega}\neq \bigcap\limits_{\omega\in\mathbf{\Omega}_0}\mathcal{E}_{\omega}\;\!$, so
the ultradistribution theories (\ref{full}) are larger than those in (\ref{restr}), but not equivalent to them.

\section{Ultradifferential operators and the main results}

For $\omega\in\mathbf{\Omega}\;\!$, let us consider the $\omega$-ultradifferentiable function spaces
$\mathcal{D}_\omega\;\!$, $\mathcal{E}_\omega\;\!$, as defined in Section 2 ($\mathcal{D}_\omega$
on page 7, and $\mathcal{E}_\omega$ as indicated in (v) on

\noindent page 5).

$\mathcal{D}_\omega$ is strict inductive limit of a sequence of nuclear Fr\'echet spaces and it is
stable under a series of elementary operations like pointwise multiplication, convolution, differentiation,
translations etc. Moreover, these operations are continuous.

$\mathcal{E}_\omega$ is a nuclear Fr\'echet space and has similar stability properties as
$\mathcal{D}_\omega\;\!$. The set $\mathcal{A}$ of all real analytic complex functions on $\mathbb{R}\;\!$,
as well as $\mathcal{D}_\omega\;\!$, are

\noindent dense subsets of $\mathcal{E}_\omega\;\!$.

The space of the $\omega$-ultradistributions is the strong dual $\mathcal{D}_\omega{\!\! '}$ of
$\mathcal{D}_\omega$ and, associating to each $\varphi\in\mathcal{E}_\omega$ the linear functional
\smallskip

\centerline{$\displaystyle \mathcal{D}_\omega\ni\psi\longmapsto\!\int\limits_{-\infty}^{+\infty}\!\varphi (s)\psi (s)
{\rm d}s\;\! ,$}
\smallskip

\noindent we obtain an inclusion map with dense range $\mathcal{E}_\omega \hookrightarrow
\mathcal{D}_\omega{\!\! '}\;\!$.

For all the above facts we send to \cite{C-Z1}, Section 2.

Let now $T : \mathcal{D}_\omega\longrightarrow \mathcal{D}_\omega$ be an {\it $\omega$-ultradifferential
operator}, that is a linear operator satisfying the condition
\smallskip

\centerline{${\rm supp}(T\varphi )\subset {\rm supp}(\varphi )\;\! ,\qquad \varphi\in\mathcal{D}_\omega\;\! .$}
\smallskip

\noindent Then $T$ is continuous and can be (uniquely) extended to a continuous linear operator
$\mathcal{E}_\omega\longrightarrow \mathcal{E}_\omega\;\!$, which will be still denoted by $T$
(\cite{C-Z1}, Theorem 2.16).

We say that an $\omega$-ultradifferential operator is {\it with constant coefficients} if it commutes with
every translation operator.
An immediate consequence of \cite{C-Z1}, Theorem 2.21 is

\begin{proposition}\label{udo-cont.coeff.}
If $f$ is an entire function of exponential type $0$ such that
\begin{equation}\label{udo-cond}
|f(it)|\leq d_0\;\! |\omega (t)^{n_0}|\;\! ,\qquad t\in\mathbb{R}
\end{equation}
for some integer $n_0\geq 1$ and real number $d_0>0\;\!$, then the formula
\begin{equation*}
\widehat{(f(D)\varphi )}(t)=f(it)\;\! \widehat{\varphi}(t)\;\! ,\qquad \varphi\in\mathcal{D}_\omega\;\! ,
t\in\mathbb{R}
\end{equation*}
defines an $\omega$-ultradifferential operator $f(D)$ with constant coefficients. Conversely,
any $\omega$-ultradifferential operator $f(D)$ with constant coefficients is of this form.
\end{proposition}
\noindent\hspace{12.3 cm}$\square$
\smallskip

If $f$ is an entire function of exponential type $0\;\!$, satisfying (\ref{udo-cond}) for some
$n_0\geq 1$ and $d_0>0\;\!$, then the  $\omega$-ultradifferential operator $f(D) : \mathcal{E}_\omega
\longrightarrow \mathcal{E}_\omega$

\noindent with constent coefficients can be extended to a continuous linear operator

\noindent $\mathcal{D}_\omega{\!\! '} \longrightarrow \mathcal{D}_\omega{\!\! '}\;\!$, which
we will still denote by $f(D)$ (see \cite{C-Z1}, discussion before 

\noindent Theorem 3.5).

Denoting by $\delta_{s_o}$ the Dirac measure concentrated at
$s_0\in\mathbb{R}\;\!$, considered an $\omega$-ultradistribution of support $\{ s_0\}\;\!$,
for each $\omega$-ultradistribution $F$ withsupport $\{ s_0\}$ there exists an entire function
as above such that $T=f(D)\delta_{s_0}$ (see \cite{C-Z1}, Theorem 3.5).
\smallskip

If $\displaystyle \omega (z)=\prod\limits_{j=1}^{\infty}\Big( 1+\frac{iz}{t_j}\Big)\;\! ,z\in\mathbb{C}\;\! ,$
where

\centerline{$\displaystyle 0<t_1\leq t_2\leq t_3\leq\;\! ...\;\! \leq +\infty\;\! ,\quad t_1<+\infty\;\! ,\quad
\sum\limits_{j=1}^{\infty}\frac 1{t_j}<+\infty\;\! ,$}
\smallskip

\noindent then, for $n\geq 1$ and $k\geq 0$ integers, we denote by $a^{\omega , n}_k$ the square
root of the coefficient of $z^k$ in the power series expansion of the entire function
\medskip

\centerline{$\displaystyle \mathbb{C}\ni z\longmapsto \big(\omega (z)\;\!\overline{\omega}(z)\big)^n=
\prod\limits_{j=1}^{\infty}\Big( 1+\frac{z^2}{t_k^2}\Big)^{\! n}$}

\noindent ($\overline{\omega}(z)$ stands here, as usual, for $\overline{\omega (\overline{z})}\,$).
We recall (see \cite{C-Z1}, page 109):
\begin{equation*}
a^{\omega , n}_k\leq a^{\omega , n+1}_k\;\! ,\qquad n\geq 1\;\! ,k\geq 0\;\! ;
\end{equation*}
\begin{equation}\label{fct-a}
\sup\limits_{p\geq 0}a^{\omega , n}_p |t|^p\leq |\omega (t)^n|\leq\sqrt{2} \sup\limits_{p\geq 0}
a^{\omega , n}_p \big|\sqrt{2}\;\! t\big|^p\;\! ,\qquad t\in\mathbb{R}\;\! .
\end{equation}
We have also, according to \cite{C-Z1}, Corollary 2.9,
\begin{equation}\label{a-logconv}
\big( a^{\omega , n}_k\big)^2\geq a^{\omega , n}_{k-1}\cdot a^{\omega , n}_{k+1}\;\! ,\qquad
n\;\! ,k\geq 1\text{ integers.}
\end{equation}
A useful consequence of (\ref{a-logconv}) is
\begin{equation}\label{compute-sup}
 \Big( \frac{a^{\omega , n}_k}{a^{\omega , n}_{k-1}}\Big)^k
\sup\limits_{p\geq 0} a^{\omega , n}_p\Big( \frac{a^{\omega , n_1}_{k-1}}{a^{\omega , n}_k}\Big)^p
=a^{\omega , n}_k\;\! ,\qquad n\;\! ,k\geq 1\text{ integers.}
\end{equation}
Indeed, since
\begin{equation*}
 \Big( \frac{a^{\omega , n}_k}{a^{\omega , n}_{k-1}}\Big)^k
\sup\limits_{p\geq 0} a^{\omega , n}_p\Big( \frac{a^{\omega , n_1}_{k-1}}{a^{\omega , n}_k}\Big)^p
=\Big( \frac{a^{\omega , n}_k}{a^{\omega , n}_{k-1}}\Big)^k
\max\bigg( 1\;\! ,\sup\limits_{p\geq 1}\prod\limits_{q=1}^p\Big( \frac{a^{\omega , n}_q}{a^{\omega , n}_{q-1}}
\cdot \frac{a^{\omega , n}_{k-1}}{a^{\omega , n}_k}\Big) \bigg)
\end{equation*}
and, by (\ref{a-logconv}),
\begin{equation*}
\frac{a^{\omega , n}_q}{a^{\omega , n}_{q-1}}\cdot \frac{a^{\omega , n}_{k-1}}{a^{\omega , n}_k}
\begin{cases}
\; \geq 1 &\!\!\! \text{for }\, q\leq k \\
\; \leq 1 &\!\!\! \text{for }\, q\geq k
\end{cases}\, ,
\end{equation*}
we deduce:
\begin{equation*}
 \Big( \frac{a^{\omega , n}_k}{a^{\omega , n}_{k-1}}\Big)^k
\sup\limits_{p\geq 0} a^{\omega , n}_p\Big( \frac{a^{\omega , n_1}_{k-1}}{a^{\omega , n}_k}\Big)^p
=\Big( \frac{a^{\omega , n}_k}{a^{\omega , n}_{k-1}}\Big)^k
\prod\limits_{q=1}^k\Big( \frac{a^{\omega , n}_q}{a^{\omega , n}_{q-1}}
\cdot \frac{a^{\omega , n}_{k-1}}{a^{\omega , n}_k}\Big) =a^{\omega , n}_k\;\! .
\end{equation*}
(\ref{compute-sup}) implies immediately:
\begin{equation}\label{compute-inf-sup}
\min\limits_{t>o}\frac 1{t^k}\;\! \sup\limits_{p\geq 0} a^{\omega , n}_p\;\! t^p=a^{\omega , n}_k\;\! ,\qquad
n\geq 1\;\! ,k\geq 0\;\! .
\end{equation}

We notice also the inequality
\begin{equation}\label{mod.omega}
\big|\omega (z)\big| =\prod\limits_{j=1}^{\infty}\Big| 1+\frac{iz}{t_j}\Big|\leq
\prod\limits_{j=1}^{\infty}\Big( 1+\frac{|z|}{t_j}\Big) =\omega (-\;\! i\;\! |z|)\;\! ,\qquad z\in\mathbb{C}\;\! .
\end{equation}

If $P$ is a polynomial with complex coefficients and $P(z)=\sum\limits_{k=0}^nc_k z^k$, then
$P(D)=\sum\limits_{k=0}^nc_k D^k$ where $D$ is the derivation operator.
The next proposition,

\noindent a variant of \cite{C-Z1}, Theorem 2.25, characterizes those
$\omega$-ultradifferential operators with constant coefficients, which can be expanded in
power series in $D$.

\begin{proposition}\label{conv.type}
Let $f$ be an entire function of exponential type $0$ such

\noindent that $($\ref{udo-cond}$)$ holds true for some
$n_0\geq 1$ and $d_0>0\;\!$, and $f(z)=\sum\limits_{k=0}^{\infty}c_k z^k$ its expansion in a
power series. Then the following statements are equivalent $:$
\begin{itemize}
\item[(i)] There exist an integer $n_1\geq 1$ and a real number $d_1>0$ such that
\begin{equation*}
\qquad |f(z)|\leq d_1\;\! \big|\;\!\omega (|z|)^{n_1}\!\big|\;\! ,\qquad z\in\mathbb{C}\;\! .
\end{equation*}
\item[(ii)] There exist an integer $n_2\geq 1$ and real numbers $L_2\;\! ,d_2>0$
such that
\begin{equation*}
\qquad |c_k|\leq d_2\;\! L_2^{k}\;\! a^{\omega , n_2}_k\;\! ,\qquad k\geq 0\;\! .
\end{equation*}
\item[(iii)] We have $\displaystyle f(D)=\sum\limits_{k=0}^{\infty}c_k D^k$,
where the series converges in the vector space of all continuous linear maps
$\mathcal{E}_\omega\longrightarrow \mathcal{E}_\omega\;\!$, endowed with the topology
of the uniform convergence on the bounded subsets of $\mathcal{E}_\omega\;\!$.
\item[(iv)] We have $\displaystyle f(D)=\sum\limits_{k=0}^{\infty}c_k D^k$,
where the series converges in the vector space of all continuous linear maps
$\mathcal{E}_\omega\longrightarrow \mathcal{E}_\omega\;\!$, endowed with the topology
of the pointwise convergence.
\end{itemize}
\end{proposition}

\begin{proof}
For (i)$\;\!\Rightarrow$(ii).
Using the Cauchy estimate, (i) and (\ref{fct-a}), we obtain for any integer $k\geq 0$ and
real $r>0\;\!$:
\medskip

\centerline{$\displaystyle |c_k|\leq \frac 1{r^k}\sup\limits_{|z|=r}|f(z)|\leq d_1\;\!  \frac 1{r^k}\sup\limits_{|z|=r}
\big|\;\!\omega (|z|)^{n_1}\!\big|
\leq \sqrt{2}\;\! d_1\;\! \frac 1{r^k}\sup\limits_{p\geq 0}a^{\omega , n_1}_p \big(\sqrt{2}\;\! r\big)^p\;\! .$}
\smallskip

\noindent Using now (\ref{compute-inf-sup}), we infer:
\begin{equation*}
\begin{split}
|c_k|\leq\;& \sqrt{2}\;\! d_1\;\! \inf\limits_{r>0}
\frac 1{r^k}\sup\limits_{p\geq 0}a^{\omega , n_1}_p \big(\sqrt{2}\;\! r\big)^p
=\sqrt{2}\;\! d_1\;\! \inf\limits_{t>0} \Big(\frac{\sqrt{2}}t\;\!\Big)^{\! k} \sup\limits_{p\geq 0}a^{\omega , n_1}_p t^p \\
=\;& \sqrt{2}\;\! d_1 \big(\sqrt{2}\,\big)^ka^{\omega , n_1}_k\;\! ,\qquad k\geq 0\;\! .
\end{split}
\end{equation*}
Thus (ii) holds with $n_2=n_1\;\! , L_2=\sqrt{2}\;\! ,d_2=\sqrt{2}\;\! d_1\;\!$.
\smallskip

For (ii)$\;\!\Rightarrow$(i). Using (ii) and the first inequality in (\ref{fct-a}), we deduce:
\begin{equation*}
\begin{split}
|f(z)|\leq\;& \sum\limits_{k=0}^{\infty}|c_k|\;\! |z|^k\leq
d_2\sum\limits_{k=0}^{\infty}a^{\omega , n_2}_k\;\! \big( L_2\;\! |z|\big)^k
=d_2\sum\limits_{k=0}^{\infty}\frac 1{2^k}\;\!a^{\omega , n_2}_k\;\! \big( 2\;\! L_2\;\! |z|\big)^k \\
\leq\;& d_2\Big(\sum\limits_{k=0}^{\infty}\frac 1{2^k}\Big) \sup\limits_{k\geq 0}
a^{\omega , n_2}_k\;\! \big( 2\;\! L_2\;\! |z|\big)^k
\leq 2\;\! d_2\;\! \big|\;\!\omega \big( 2\;\! L_2\;\! |z|\big)^{n_2}\!\big|\;\! .
\end{split}
\end{equation*}
Choosing some integer $m\geq 2\;\! L_2$ and using (\ref{seminorms}), we obtain
\begin{equation*}
|f(z)|\leq 2\;\! d_2\;\! \big|\;\!\omega \big( m\;\! |z|\big)^{n_2}\!\big|\leq
2\;\! d_2\;\! \big|\;\!\omega (|z|)^{m^2 n_2}\big|\;\! ,
\end{equation*}
hence (i) holds with $n_1=m^2n_2$ and $d_1=2\;\! d_2\;\!$.
\smallskip

Implication (ii)$\;\!\Rightarrow$(iii) follows by \cite{C-Z1}, Proposition 2.24, and
implication (iii)$\;\!\Rightarrow$(iv) is trivial.
\smallskip

Finally, for the proof of (iv)$\;\!\Rightarrow$(ii) we adapt the proof of (iv)$\;\!\Rightarrow$(ii)
in \cite{C-Z1}, Theorem 2.25 as follows.

(iv) implies that the sequence $(c_k\;\! D^k\varphi )_{k\geq 0}=(c_k\;\! \varphi^{(k)} )_{k\geq 0}$
converges in $\mathcal{E}_{\omega}$ to $0$ for every $\varphi\in \mathcal{E}_{\omega}\;\!$.
Therefore the sequence
$\mathcal{E}_{\omega}\ni\varphi\longmapsto c_k\varphi^{(k)}(0)\;\! , k\geq 0\;\! ,$ is pointwise
convergent to $0$ in $\mathcal{E}_{\omega}{\!\! '}\;\!$, in particular it is pointwise bounded.
Since $\mathcal{E}_{\omega}$ is a Fr\'echet space, and hence barrelled, if follows
that the above sequence in $\mathcal{E}_{\omega}{\!\! '}$ is equicontinuous (see e.g.
\cite{Bou}, Ch. III, \S 4, Section 1).

Recalling that the topology of $\mathcal{E}_{\omega}$ is defined by the semi-norms
\smallskip

\centerline{$\displaystyle r_{\omega ,L,n}^{K} : \mathcal{E}_{\omega}\ni\varphi\longmapsto
\sup\limits_{p\geq 0}\big( L^pa^{\omega , n}_p\sup\limits_{s\in K} |\varphi^{(p)}(s)|\big)\;\! ,$}
\smallskip

\noindent where $K\subset\mathbb{R}$ is compact, $L>0$ and $n\geq 1$ is an integer (see
\cite{C-Z1}, Definition V on page 110), we deduce the existence of some $K\;\! ,L\;\! ,n$
and of a constant $d>0$ such that
\begin{equation*}
|c_k\varphi^{(k)}(0)|\leq d\cdot r_{\omega ,L,n}^{K}(\varphi )\;\! ,\qquad k\geq 0\;\! ,\varphi\in
\mathcal{E}_{\omega}\;\! .
\end{equation*}
Applying this inequality to $\varphi =e^{i\;\!\alpha\;\!\cdot}, \alpha >0\;\!$, we obtain
\begin{equation*}
|c_k|\cdot \alpha^k\leq d\cdot\sup\limits_{p\geq 0}\big( L^pa^{\omega , n}_p \alpha^p\big)\;\! ,
\qquad k\geq 0\;\! ,\alpha >0\;\! .
\end{equation*}
Therefore, using (\ref{compute-inf-sup}), we infer:
\begin{equation*}
|c_k|\leq d \inf\limits_{\alpha >0}\frac 1{\alpha^k} \sup\limits_{p\geq 0}a^{\omega , n}_p
(L\;\!\alpha )^p =d\;\! \inf\limits_{t>0} \Big(\frac Lt\;\!\Big)^{\! k}\! \sup\limits_{p\geq 0}a^{\omega , n}_p t^p
=d\;\! L^ka^{\omega , n}_k\;\! ,\quad k\geq 0\;\! .
\end{equation*}
In other words, (ii) holds with $n_2=n\;\! , L_2=L\;\! ,d_2=d\;\!$.

\end{proof}

If $f$ is an entire function satisfying the equivalent conditions in Proposition \ref{conv.type},
then we will say that the $\omega$-ultradifferential operator with constant coefficients
$f(D)$ is {\it of convergence type}.
\smallskip

Proposition \ref{conv.type} enables to prove a description of those $\omega\in\mathbf{\Omega}\;\!$,
for which every $\omega$-ultradifferential operator with constant coefficients is of convergence type.
This description is essentially \cite{C-Z1}, Theorem 2.25.

\begin{corollary}\label{strong-n.q.}
The following statements concerning $\omega\in\mathbf{\Omega}$ are equivalent$:$
\begin{itemize}
\item[(j)] There exist an integer $n_1\geq 1$ and a real number $d_1>0$ such that
\begin{equation*}
\qquad \big|\omega (-it)\big|\leq d_1\;\! \big|\;\!\omega (t)^{n_1}\!\big|\;\! ,\qquad t\in\mathbb{R}\;\! .
\end{equation*}
\item[(jj)] There exist an integer $n_2\geq 1$ and a real number $d_2>0$ such that
\begin{equation*}
\qquad \big|\omega (z)\big|\leq d_2\;\! \big|\;\!\omega (|z|)^{n_2}\!\big|\;\! ,\qquad z\in\mathbb{C}\;\! .
\end{equation*}
\item[(jjj)] The $\omega$-ultradifferential operator with constant coefficients $\omega (-iD)$

\noindent is of convergence type.
\item[(jw)] Every $\omega$-ultradifferential operator with constant coefficients is of

\noindent convergence type.
\end{itemize}
\end{corollary}

\begin{proof}
(j)$\;\!\Rightarrow$(jj) follows easily by using (\ref{mod.omega}):

\centerline{$\big|\omega (z)\big|\overset{(\ref{mod.omega})}{\leq} \omega (-\;\! i\;\! |z|)
\overset{({\rm j})}{\leq} d_1\;\! \big|\;\!\omega (|z|)^{n_1}\;\! ,z\in\mathbb{C}\;\! .$}
\medskip

\noindent On the other hand, implication (jj)$\;\!\Rightarrow$(j) is trivial:

\centerline{$\big|\omega (-it)\big|\overset{({\rm jj})}{\leq}d_2\;\! \big|\;\!\omega (|-it|)^{n_2}\!\big|
=d_2\;\! \big|\;\!\omega (|t|)^{n_2}\!\big| =d_2\;\! \big|\;\!\omega (t)^{n_2}\!\big|\;\! ,\qquad
t\in\mathbb{R}\;\! .$}
\medskip

\noindent Thus (j)$\;\!\Leftrightarrow$(jj).
\smallskip

Next, equivalence (jj)$\;\!\Leftrightarrow$(jjj) is an immediate consequence of the definition
of the convergence type by using condition (i) in Proposition \ref{conv.type}, while implication
(jw)$\;\!\Rightarrow$(jjj) is trivial. Thus it remains only to prove, for example, (jj)$\;\!\Rightarrow$(jw).
\smallskip

For let us assume that (jj) is satisfied and $f$ is an arbitrary entire function of exponential
type $0\;\!$, satisfying (\ref{udo-cond}).

Denoting, for convenience,
\medskip

\centerline{$\rho (z):=d_0\cdot\omega (z)^{n_0}\;\! ,\qquad z\in\mathbb{C}\;\! ,$}
\medskip

\noindent $\rho$ is an entire function of exponential type $0\;\!$, which has no zeros in the
open lower half-plane. Therefore, using the terminology of \cite{L}, Chapter VII, \S 4, $\rho$
is an {\it entire function of class} $P$. Since $f(i\,\cdot\;\! )$ is an entire function of exponential
type $0$ and, by (\ref{udo-cond}),
\medskip

\centerline{$|f(it)|\leq |\rho (t)|\;\! ,\qquad t\in\mathbb{R}\;\! ,$}
\medskip

\noindent applying  \cite{L}, Chapter IX, \S 4, Lemma 1, we obtain:
\medskip

\centerline{$\big| f(iz)\big|\leq \big| \rho (z)\big|\text{ and }\big| f(i\;\!\overline{z})\big|\leq
\big| \rho (z)\big|\text{ for all }z\in\mathbb{C}\text{ with }{\rm Im}z\leq 0\;\! ,$}
\medskip

\noindent that is
\begin{equation}\label{first}
\big| f(iz)\big|\leq
\begin{cases}
\; d_0\;\! \big|\;\! \omega (z)^{n_0}\big| &\!\!\! \text{for }\, z\in\mathbb{C}\;\! ,{\rm Im}z\leq 0 \\
\; d_0\;\! \big|\;\! \omega (\overline{z})^{n_0}\big| &\!\!\! \text{for }\, z\in\mathbb{C}\;\! ,{\rm Im}z\geq 0
\end{cases}
\end{equation}
On the other hand, (jj) yields for every $z\in\mathbb{C}$
\begin{equation}\label{second}
\big|\;\! \omega (z)\big|\leq d_2\;\! \big|\;\!\omega (|z|)^{n_2}\!\big|\text{ and }\big|\;\!
\omega (\overline{z})\big|\leq d_2\;\! \big|\;\!\omega (|\overline{z}|)^{n_2}\!\big| =
d_2\;\! \big|\;\!\omega (|z|)^{n_2}\!\big|\;\! .
\end{equation}
Now, by (\ref{first}) and (\ref{second}) we deduce:
\medskip

\centerline{$\big| f(iz)\big|\leq d_0\;\! (d_2)^{n_0} \big|\;\!\omega (|z|)^{n_2\cdot n_0}\big|\;\! ,\qquad
z\in\mathbb{C}\;\! .$}
\medskip

Consequently condition (i) in Proposition \ref{conv.type} holds true with $n_1=n_2\cdot n_0$ and
$d_1=d_0\;\! (d_2)^{n_0}\;\!$, and we conclude that the $\omega$-ultradifferential operator with
constant coefficients $f(D)$ is of convergence type.

\end{proof}

Following \cite{C-Z1}, Definition XI, we will say that $\omega$ satisfies the {\it strong non-quasianalyticity
condition} whenever it fulfills the equivalent conditions  in Corollary \ref{strong-n.q.}.

\begin{remark}\label{char.strong-n.q.}
{\it If $\displaystyle \omega (z)=\prod\limits_{j=1}^{\infty}\Big( 1+\frac{iz}{t_j}\Big)\;\! ,z\in\mathbb{C}\;\! ,$
where

\centerline{$\displaystyle 0<t_1\leq t_2\leq t_3\leq\;\! ...\;\! \leq +\infty\;\! ,\quad t_1<+\infty\;\! ,\quad
\sum\limits_{j=1}^{\infty}\frac 1{t_j}<+\infty\;\! ,$}
\smallskip

\noindent then, in order that $\omega$ satisfy the strong non-quasianalyticity condition, a necessary
condition is}
\smallskip

\centerline{$\displaystyle \sum\limits_{j=1}^{\infty}\frac{\ln j}{t_j}<+\infty$}
\smallskip

\noindent ({\it see \cite{C-Z1}, Corollary 1.9}), {\it while a sufficient condition is the existence of a constant
$c>0$ such that}
\smallskip

\centerline{$\displaystyle \sum\limits_{j=k}^{\infty}\frac 1{t_j}\leq c\;\! \frac k{t_k}\;\! ,\qquad k\geq 1$}
\smallskip

\noindent ({\it see \cite{K}, Proposition 4.6 or \cite{C-Z1}, comments after Proposition 5.15}). {\it If we
assume also
\smallskip

\centerline{$\displaystyle 0<t_1\leq \frac{t_2}2\leq \frac{t_3}3\leq\;\! ...\;\! ,$}
\medskip

\noindent then $\omega$ satisfies the strong non-quasianalyticity condition if and only if there exists
a constant $c>0$ such that}
\medskip

\centerline{$\displaystyle \sum\limits_{j=k}^{\infty}\frac 1{t_j}\leq c\;\! \frac k{t_k}\Big( 1+
\ln \frac{t_k}{(t_1\;\! ...\;\! t_k)^{1/k}}\Big)\;\! ,\qquad k\geq 1$}
\smallskip

\noindent ({\it see \cite{C-Z1}, Proposition 5.15}).
\end{remark}

Central issue in the theory of $\omega$-ultradifferential operators with constant coefficients
$f(D) : \mathcal{D}_\omega{\!\! '} \longrightarrow \mathcal{D}_\omega{\!\! '}$ is the characterization
of its surjectivity, that is of the existence of a solution $X\in \mathcal{D}_\omega{\!\! '}$ of the
equation $f(D)X=F$ for each $F\in \mathcal{D}_\omega{\!\! '}\;\!$, in terms of $f$.
A surjectivity criterion was proved by I. Cior\u{a}nescu in \cite{Ci}, Proposition 2.4 and Theorem 3.4$\;\! :$

\begin{proposition}\label{criterion}
For $f$ an entire function of exponential type $0$ such that $($\ref{udo-cond}$)$ holds true for
some $n_0\geq 1$ and $d_0>0\;\!$., the following statements are equivalent$\;\! :$
\begin{itemize}
\item[(i)] There exists some $E\in \mathcal{D}_\omega{\!\! '}$ such that $f(D)E=\delta_0\;\!$.
\item[(ii)] $f(D) : \mathcal{D}_\omega{\!\! '} \longrightarrow \mathcal{D}_\omega{\!\! '}$ is surjective,
that is $f(D)\mathcal{D}_\omega{\!\! '}=\mathcal{D}_\omega{\!\! '}\;\!$.
\item[(iii)] there are constants $c\;\! ,c'>0$ such that
\begin{equation*}
\sup\limits_{\substack{s\in\mathbb{R} \\ |s-t|\leq c \ln |\omega (t)|+c'}} \ln |f(s)|\geq -\;\! c \ln |\omega (t)|
-c'\;\! ,\qquad t\in\mathbb{R}\;\! .
\end{equation*}
\end{itemize}
\hspace{12 cm}$\square$
\end{proposition}

If $f$ satisfies the equivalent conditions of Proposition \ref{criterion}, then, following \cite{Ch},
D\'efinition III.1-4, and by abuse of language, we will say that $f(D)$ is {\it invertible in}
$\mathcal{D}_\omega{\!\! '}\;\!$.
\smallskip

If $\rho\in\mathbf{\Omega}$ and

\centerline{$|\omega (t)|\leq c\;\! |\rho (t)|\;\! ,\qquad t\in\mathbb{R}$}
\medskip

\noindent for some $c>0\;\!$, then $\mathcal{D}_\rho\subset\mathcal{D}_\omega\;\!$, where the
inclusion is continuous and with dense range, Consequently also $\mathcal{D}_\omega{\!\! '}
\subset \mathcal{D}_\rho{\!\! '}\;\!$, where the inclusion is continuous and with dense range.
Any $\omega$-ultradifferential operator with constant coefficients $f(D)$ is clearly also a
$\rho$-ultradifferential operator with constant coefficients, hence we can consider the problem of
the invertibility of $f(D)$ in $\mathcal{D}_\rho{\!\! '}\;\!$. We notice that is $f(D)$ is of convergence
type as $\omega$-ultradifferential operator, then it is of convergence type also as
$\rho$-ultradifferential operator.
\smallskip

The main goal of this paper is to give an exact answer to the question: for which $\omega\in
\mathbf{\Omega}$ is every $\omega$-ultradifferential operator with constant coefficients and of convergence
type, $\rho$-invertible for some $\rho\in\mathbf{\Omega}$ with $\omega\leq c\;\!\rho\;\!$, where

\noindent $c>0$ is a constant?

A sufficient condition for this was already found in \cite{C-Z2}, Proposition 2.7, namely
\begin{equation}\label{w-str}
\int\limits_1^{+\infty}\frac{\ln |\omega (t)|}{t^2} \ln \frac t{\ln |\omega (t)|}\;\!{\rm d}t<+\infty\;\! .
\end{equation}
Let us call condition (\ref{w-str}) the {\it mild strong non-quasianalyticity condition}.
This denomination is justified by the fact that (\ref{w-str}) is implied by the strong non-quasianalyticity
property. More precisely, we have$\;\! :$

\begin{proposition}\label{criteria1}
For $\displaystyle \omega (z)=\prod\limits_{j=1}^{\infty}\Big( 1+\frac{iz}{t_j}\Big)\;\! ,z\in\mathbb{C}\;\! ,$
where

\centerline{$\displaystyle 0<t_1\leq t_2\leq t_3\leq\;\! ...\;\! \leq +\infty\;\! ,\quad t_1<+\infty\;\! ,\quad
\sum\limits_{j=1}^{\infty}\frac 1{t_j}<+\infty\;\! ,$}
\smallskip

\noindent we have
\begin{equation*}
\begin{split}
&\omega\text{ satisfies the strong non-quasianalyticity condition} \\
\Longrightarrow\;& \sum\limits_{j=1}^{\infty}\frac{\ln j}{t_j}<+\infty\Longleftrightarrow
\sum\limits_{j=1}^{\infty}\frac{\ln t_j}{t_j}<+\infty \\
\Longrightarrow\;& \sum\limits_{j=1}^{\infty}\frac{\displaystyle \ln \frac{t_j}j}{t_j}<+\infty
\Longleftrightarrow \begin{array}{l}\omega\text{ satisfies the mild strong} \\
\text{non-quasianalyticity condition.}\end{array}
\end{split}
\end{equation*}
\end{proposition}

\begin{proof}
The first implication was already pointed out in Remark \ref{char.strong-n.q.}.

A proof of equivalence $\displaystyle \sum\limits_{j=1}^{\infty}\frac{\ln j}{t_j}<+\infty
\Longleftrightarrow\sum\limits_{j=1}^{\infty}\frac{\ln t_j}{t_j}<+\infty$ was given in the
comments after \cite{C-Z1}, Corollary 1.9 (Page 92).

The second implication is trivial, while the last equivalence is (i)$\Leftrightarrow$(iv) in
\cite{C-Z2}, Lemma 2.1.

\end{proof}

We will need the next calculus lemma$\;\! :$

\begin{lemma}\label{calculus}
Let $\alpha\;\! ,\gamma : (0\;\! ,+\infty )\longrightarrow (0\;\! ,+\infty )$ be two functions such that
\begin{equation*}
\frac{\gamma (t)}{\alpha (t)}\geq e\;\! ,\qquad t>0\;\! .
\end{equation*}
\begin{itemize}
\item[(i)] If $\alpha$ and $\gamma$ are increasing, then also the function
\begin{equation}\label{output}
(0\;\! ,+\infty )\ni t\longmapsto \alpha (t) \ln \frac{\gamma (t)}{\alpha (t)}\in (0\;\! ,+\infty )
\end{equation}
is increasing.
\item[(ii)] If $\alpha$ and $\gamma$ are twice differentiable and concave, then also the
function {\rm (\ref{output})} is twice differentiable and concave.
\end{itemize}
\end{lemma}

\begin{proof}
{\bf For (i).} Assume that $\alpha\;\! ,\gamma$ are increasing and let $0<t_1<t_2$ be arbitrary.
Then
\begin{equation*}
\alpha (t_1)\;\! \ln \frac{\gamma (t_1)}{\alpha (t_1)}\leq
\alpha (t_1)\;\! \ln \frac{\gamma (t_2)}{\alpha (t_1)}=\gamma (t_2)\Big(
\frac{\alpha (t_1)}{\gamma (t_2)} \ln \frac{\gamma (t_2)}{\alpha (t_1)}\Big)\;\! .
\end{equation*}
Since $\displaystyle \frac{\gamma (t_2)}{\alpha (t_1)}\geq \frac{\gamma (t_2)}{\alpha (t_2)}\geq e$
and $\displaystyle x\longmapsto \frac 1x \ln x$ is decreasing on $[\;\! e\;\! ,+\infty )\;\!$, we get
\begin{equation*}
\alpha (t_1)\;\! \ln \frac{\gamma (t_1)}{\alpha (t_1)}\leq\gamma (t_2)\Big(
\frac{\alpha (t_2)}{\gamma (t_2)} \ln \frac{\gamma (t_2)}{\alpha (t_2)}\Big) =
\alpha (t_2)\;\! \ln \frac{\gamma (t_2)}{\alpha (t_2)}\;\! .
\end{equation*}

{\bf For (ii).} Assume that $\alpha\;\! ,\gamma$ are twice differentiable and concave, hence
$\alpha'' ,\gamma''\leq 0\;\!$. Function (\ref{output}) is clearly twice differentiable, its
first derivative at $t>0$ is
\begin{equation*}
\alpha'(t)\ln \frac{\gamma (t)}{\alpha (t)}+\frac{\gamma'(t)\alpha (t)-\alpha'(t)\gamma (t)}{\gamma (t)}\;\! ,
\end{equation*}
while its second derivative at $t>0$ is
\begin{equation*}
\alpha''(t)\Big( \ln \frac{\gamma (t)}{\alpha (t)}-1\Big) +\gamma''(t)\frac{\alpha (t)}{\gamma (t)}
-\frac{\;\!\big( \gamma'(t)\alpha (t)-\alpha'(t)\gamma (t)\big)^2}{\alpha (t) \gamma (t)^2}\;\! .
\end{equation*}
Since $\alpha''(t)\;\!\gamma''(t)\leq 0$ and $\displaystyle ln \frac{\gamma (t)}{\alpha (t)}-1
\geq \ln e-1=0\;\!$, the second derivative is $\leq 0$ at all $t>0\;\!$.

\end{proof}

The next theorem is a slightly extended version of  \cite{C-Z2}, Theorem 2.2. For its proof
we adapted the proof of \cite{C-Z2}, Theorem 2.2.

\begin{theorem}\label{min}
Let us assume that $\omega\in\mathbf{\Omega}$ satisfies the mild strong non-

\noindent quasianalyticity condition. Then there exists some $\rho\in\mathbf{\Omega}$ satisfying
\begin{equation}\label{major}
|\omega (t)|\leq c_0\;\! |\rho (t)|\;\! ,\qquad t\in\mathbb{R}
\end{equation}
with $c_0>0$ a constant, such that$\;\! :$
\smallskip

If $f$ is an entire function and
\begin{equation}\label{min-1}
|f(z)|\leq d_0\;\!\big|\;\!\omega (|z|)^{n_0}\!\big|\;\! ,\qquad z\in\mathbb{C}
\end{equation}
for some integer $n_0\geq 1$ and $d_0>0\;\!$, then there are constants $c\;\! ,c'>0$
such that
\begin{equation}\label{min-2}
\sup\limits_{\substack{s\in\mathbb{R} \\ |s-t|\leq c \ln |\rho (t)|+c'}} \ln |f(s)|\geq -\;\! c \ln |\rho (t)|
-c'\;\! ,\qquad t\in\mathbb{R}\;\! .
\end{equation}

Moreover, if $\omega\in\mathbf{\Omega}_0\;\!$, then we can choose $\rho\in\mathbf{\Omega}_0\;\!$.
\end{theorem}

\begin{proof}
In the case of a general $\omega\in\mathbf{\Omega}\;\!$, let $\alpha$ denote the function
\begin{equation*}
(0\;\! ,+\infty )\ni t\longmapsto \ln |\omega (t)|\in (0\;\! ,+\infty )\;\! .
\end{equation*}

In the case of $\omega\in\mathbf{\Omega}_0$ we need for $\alpha$ an infinitely differentiable,
increasing, concave function satisfying
\begin{equation}\label{msnq-majorant}
\int\limits_1^{+\infty}\!\frac{\alpha (t)}{t^2}\;\!{\rm d}t<+\infty\, ,\quad
\int\limits_1^{+\infty}\!\frac{\alpha (t)}{t^2}\ln \frac t{\alpha (t)}\;\!{\rm d}t<+\infty
\end{equation}
and $\ln |\omega (t)|\leq\alpha (t)\;\! ,t>0\;\! .$
To obtain it, let
\medskip

\centerline{$\displaystyle 0<t_1\leq \frac{t_2}2\leq \frac{t_3}3\leq\;\! ...\;\! \leq +\infty\;\! ,\quad
t_1<+\infty\;\! ,\quad \sum\limits_{j=1}^{\infty}\frac 1{t_j}<+\infty$}

\noindent be such that $\displaystyle \omega (z)=\prod\limits_{j=1}^{\infty}\Big( 1+\frac{iz}{t_j}\Big)\;\! ,\,
z\in\mathbb{C}\;\!$, and set

\centerline{$\displaystyle \alpha (t):=\ln 3 + 2 \ln\Big(1+\sum\limits_{k=1}^{\infty}
\frac{(4\;\! t)^k}{t_1\;\! ...\;\! t_k}\Big)\;\! ,\qquad t>0\;\! .$}
\smallskip

\noindent Then $(0\;\! ,+\infty )\ni t\longmapsto\alpha (t)\in (0\;\! ,+\infty )$ is infinitely differentiable,
increasing and, according to Lemma \ref{concave-explicite}, concave.
On the other hand, by \cite{C-Z1}, Lemma 1.7, we have $\ln |\omega (t)|\leq\alpha (t)\;\! ,t>0\;\! .$
Finally, since by \cite{C-Z1}, Lemma 1.7,
\begin{equation*}
\alpha (t)\leq \ln 3 + 2\big( \ln 4+2 \ln |\omega (8\;\! t)|\big) =\ln (48)+4 \ln |\omega (8\;\! t)|\;\! ,\qquad
t>0\;\! ,
\end{equation*}
and $\omega$ satisfies the mild strong non-quasianalyticity condition, Lemma \ref{msnq-permanence1}

\noindent yields (\ref{msnq-majorant}).
\smallskip

An inspection of the proof of \cite{C-Z2}, Corollary 1.2 shows that there exists a constant $\lambda >0$
such that
\begin{equation*}
\frac{1+t}{\;\!\lambda\;\! \alpha (2et)}>8\;\! e\;\! ,\qquad t>0\;\! ,
\end{equation*}
and the function
\begin{equation*}
\beta : (0\;\! ,+\infty )\ni t\longmapsto
6\;\!\alpha (2et)\ln\frac{1+t}{\;\!\lambda\;\! \alpha (2et)}+8\sum\limits_{j=1}^{\infty}
\frac{\;\!\alpha (2^jet)}{4^j}\in (0\;\! ,+\infty )\;\! ,
\end{equation*}
which is, according to Lemma \ref{calculus}, increasing and, in the case of $\omega\in\mathbf{\Omega}_0\;\!$,

\noindent also concave,
satisfies $\displaystyle \int\limits_1^{+\infty}\!\frac{\beta (t)}{t^2}\;\!{\rm d}t<+\infty$
and has the property$\;\! :$
\smallskip

If $f$ is any entire function satisfying
\begin{equation*}
\ln |f(z)|\leq d\;\! \alpha (|z|)+d'\;\! ,\qquad z\in\mathbb{C}
\end{equation*}
for some $d\;\! ,d'>0\;\!$, then there exist constants $c_1\;\! ,c_1'>0$ such that
\begin{equation*}
\sup\limits_{t\leq r\leq t+c_1 \alpha (t)}\;\! \inf\limits_{|z|=r}\ln |f(z)|\geq -c_1\;\!\beta(t)-c_1'\;\! ,
\qquad t>0\;\! .
\end{equation*}

We notice that $\beta (t)\geq 6\;\!\alpha (2et)\ln (8\;\! e)>\alpha (t)$ for all $t>0\;\!$.
\smallskip

Now, according to a result of O. I. Inozemcev and V. A. Marcenko (\cite{I-M}, Theorem 1, see
also \cite{C-Z1}, Theorem 1.6), there exist $\rho\in\mathbf{\Omega}$ and a constant $d_1>0$ such that
\begin{equation*}
\beta (t)\leq \ln |\rho (t)|+d_1\;\! ,\qquad t>0\;\! .
\end{equation*}
Moreover, in the case of $\omega\in\mathbf{\Omega}_0\;\!$, when $\beta$ is increasing and concave,
Theorem \ref{concave} ensures that $\rho$ can be chosen belonging to $\mathbf{\Omega}_0\;\!$.

Since
\smallskip

\centerline{$|\omega (t)|\leq e^{\alpha (t)}<e^{\beta (t)}\leq e^{\ln |\rho (t)|+d_1}=e^{d_1}|\rho (t)|\;\! ,
\qquad t<0\;\! ,$}
\medskip

\noindent (\ref{major}) holds true with $c_0=e^{d_1}$.

Let $f$ be an entire function satisfying (\ref{min-1}). Then
\medskip

\centerline{$\ln |f(z)|\leq n_0\;\!\ln \big|\omega (|z|)\big| +\ln d_0\leq
n_0\;\!\alpha (|z|)+\ln d_0\;\! ,\qquad z\in\mathbb{C}\;\! .$}
\medskip

\noindent By the choice of $\beta$ there exist then constants $c_1\;\! ,c_1'>0$ such that
\begin{equation*}
\sup\limits_{t\leq r\leq t+c_1\;\! \alpha (t)} \inf\limits_{|z|=r}\ln |f(z)|\geq -c_1\;\!\beta(t)-c_1'\;\! ,\qquad
t>0\;\! .
\end{equation*}
It follows for every $t\in\mathbb{R}$
\begin{equation*}
\begin{split}
\sup\limits_{\substack{s\in\mathbb{R} \\ |s-t|\leq c_1 \ln |\rho (t)|+c_1'+c_1 d_1}} \ln |f(s)|\geq\;&
\sup\limits_{\substack{s\in\mathbb{R} \\ |s-t|\leq c_1 \beta (t)|+c_1'}} \ln |f(s)| \\
\geq\;&\sup\limits_{|t|\leq r\leq |t|+c_1\;\! \alpha (t)}\;\! \inf\limits_{|z|=r}\ln |f(z)| \\
\geq\;& -c_1\;\!\beta(|t|)-c_1' \\
\geq\;& -c_1\;\! \big|\rho (|t|)\big|-c_1'-c_1 d_1 \\
=\;& -\;\! c_1\;\! |\rho (t)|-c_1'-c_1 d_1\;\! .
\end{split}
\end{equation*}
Consequently (\ref{min-2}) holds true with $c=c_1$ and $c'=c_1'+c_1 d_1\;\!$.

\end{proof}

Theorem \ref{min} implies that mild strong non-quasianalyticity of $\omega\in\mathbf{\Omega}$ is a sufficient
condition in order that every $\omega$-ultradifferential operator with constant coefficients and of
convergence type be invertible in $\mathcal{D}_\rho{\!\! '}$ for some $\rho\in\mathbf{\Omega}$ satisfying
$|\omega (t)|\leq c_0\;\! |\rho (t)|\;\! ,t\in\mathbb{R}\;\! ,$ with $c_0>0$ a constant.
This is the statement of \cite{C-Z2}, Proposition 2.7$\;\! :$

\begin{theorem}\label{surjectivity}
If $\omega\in\mathbf{\Omega}$ is satisfying the mild strong non-quasianalyticity

\noindent condition, then
there exist $\rho\in\mathbf{\Omega}$ and a constant $c_0>0$ with
\medskip

\centerline{$|\omega (t)|\leq c_0\;\! |\rho (t)|\;\! ,\qquad t\in\mathbb{R}\;\! ,$}
\medskip

\noindent such that every $\omega$-ultradifferential operator with constant coefficients and of
convergence type is invertible in $\mathcal{D}_\rho{\!\! '}\;\!$.

Moreover, if $\omega\in\mathbf{\Omega}_0\;\!$, then we can choose $\rho\in\mathbf{\Omega}_0\;\!$.
\end{theorem}

\begin{proof}
Choose $\rho$ and $c_0$ as in Theorem \ref{min}.

According to Propositions \ref{udo-cont.coeff.} and \ref{conv.type}, every $\omega$-ultradifferential
operator with constant coefficients and of convergence type is of the form $f(D)$ with $f$ an entire
function satisfying condition (i) in Proposition \ref{conv.type}. By the choice of $\rho$ and $c_0\;\!$,
there exist constants $c\;\! ,c'>0$ such that (\ref{min-2}) is satisfied.

Applying now Proposition \ref{criterion}, we conclude that $f(D)$ is invertible in
$\mathcal{D}_\rho{\!\! '}\;\!$.

\end{proof}

The main result of this paper is the following theorem, which shows that Theorem \ref{min}
is sharp. It will be proved in Section 6.

\begin{theorem}\label{no-min}
Let us assume that $\omega\in\mathbf{\Omega}$ does not satisfy the mild strong
non-quasianalyticity condition, that is such that
\begin{equation*}
\int\limits_1^{+\infty}\frac{\ln |\omega (t)|}{t^2}\;\! \ln \frac t{\ln |\omega (t)|}\;\!{\rm d}t=+\infty\;\! .
\end{equation*}
Then there exists an entire function $f$ such that
\begin{equation}\label{dominated}
|f(z)|\leq \big|\omega (|z|)\big|^2\;\! ,\qquad z\in\mathbb{C}
\end{equation}
but for no increasing $\beta : (0\;\! ,+\infty )\longrightarrow (0\;\! ,+\infty )$ with
\smallskip

\centerline{$\displaystyle \int\limits_1^{+\infty}\frac{\beta (t)}{t^2}\;\!{\rm d}t<+\infty$}

\noindent can hold the condition
\begin{equation}\label{minimum}
\sup\limits_{\substack{s\in\mathbb{R} \\ |s-t|\leq \beta (t)}} \ln |f(s)|\geq -\;\! \beta (t)\;\! ,\qquad
t>0\;\! .
\end{equation}
\end{theorem}
\medskip

Using Theorem \ref{no-min}, we infer that also Theorem \ref{surjectivity} is sharp$\;\! :$

\begin{theorem}\label{no-surjectivity}
Let us assume that $\omega\in\mathbf{\Omega}$ does not satisfy the mild strong non-quasianalyticity
condition. Then there exists some $\omega$-ultradifferential

\noindent operator with constant coefficients and of convergence type, which is not

\noindent invertible in $\mathcal{D}_\rho{\!\! '}$ for any $\rho\in\mathbf{\Omega}$ satisfying
\medskip

\centerline{$|\omega (t)|\leq c_0\;\! |\rho (t)|\;\! ,\qquad t\in\mathbb{R}$}
\medskip

\noindent for some constant $c_0>0\;\!$.
\end{theorem}

\begin{proof}
Let $f$ be an entire function $f$ as in Theorem \ref{no-min}.
Then, according

\noindent to Propositions \ref{udo-cont.coeff.} and \ref{conv.type}, we can consider the
$\omega$-ultradifferential operator with constant coefficients $f(D)\;\!$, and it is of
convergence tyoe.

If it would exist some $\rho\in\mathbf{\Omega}$ satisfying
\medskip

\centerline{$|\omega (t)|\leq c_0\;\! |\rho (t)|\;\! ,\qquad t\in\mathbb{R}$}
\medskip

\noindent with $c_0>0$ a constant, such that $f(D)$ is invertible in $\mathcal{D}_\rho{\!\! '}\;\!$,
then Proposition \ref{criterion} would imply the existence of constants $c\;\! ,c'>0$ such that
\begin{equation*}
\sup\limits_{\substack{s\in\mathbb{R} \\ |s-t|\leq c \ln |\rho (t)|+c'}} \ln |f(s)|\geq -\;\! c \ln |\rho (t)|
-c'\;\! ,\qquad t\in\mathbb{R}\;\! .
\end{equation*}
But this is not possible because $\beta : (0\;\! ,+\infty )\ni t\longmapsto c \ln |\rho (t)|+c'\in
(0\;\! ,+\infty )$ would be a function with $\displaystyle \int\limits_1^{+\infty}\frac{\beta (t)}{t^2}\;\!
{\rm d}t<+\infty$ (see e.g. \cite{I-M}, Theorem 1 or \cite{C-Z1}, Theorem 1.6) such that
\begin{equation*}
\sup\limits_{\substack{s\in\mathbb{R} \\ |s-t|\leq \beta (t)}} \ln |f(s)|\geq -\;\! \beta (t)\;\! ,\qquad
t\in\mathbb{R}\;\! ,
\end{equation*}
in contradiction with the choice of $f$.

\end{proof}

\section{On the non-quasianalyticity condition}

For sake of convenience, we will say that a Lebesgue measurable function
$\alpha : (0\;\! ,+\infty )\longrightarrow (0\;\! ,+\infty )$ satisfies the {\it non-quasianalyticity condition} if
\smallskip

\centerline{$\displaystyle \int\limits_1^{+\infty}\!\frac{\alpha (t)}{t^2}\;\!{\rm d}t<+\infty\;\! .$}
\smallskip

\noindent This denomination is suggested by the classical Denjoy-Carleman Theorem
(se e.g. \cite{Man}, 4.1.III) in which non-quasianalyticity is characterized by this condition.

Examples of functions satisfying the non-quasianalyticity condition$\;\! :$

\begin{remark}\label{explicite}
{\it If $\displaystyle 0<t_1\leq t_2\leq t_3\leq\;\! ...\;\! \leq +\infty\, ,\, t_1<+\infty\;\! ,
\,\sum\limits_{j=1}^{\infty}\frac 1{t_j}<+\infty\;\! ,$

\noindent then the increasing functions
\begin{itemize}
\item[(1)] $(0\;\! ,+\infty )\ni t\longmapsto n(t)$ with $n(t)$ the number of the elements of the set
$\{ k\geq 1 ; t_k\leq t\}$ $(\;\!$the distribution function of the sequence $(t_j)_{j\geq 1}\;\! )$,
\item[(2)] $\displaystyle (0\;\! ,+\infty )\ni t\longmapsto \alpha (t):=\ln\Big(1+\sum\limits_{k=1}^{\infty}
\frac{t^k}{t_1\;\! ...\;\! t_k}\Big)\in (0\;\! ,+\infty )\, ,$
\item[(3)] $\displaystyle (0\;\! ,+\infty )\ni t\longmapsto N(t):=
\ln \max \Big( 1\;\! ,\sup\limits_{k\geq 1}\frac{t^k}{t_1\;\! t_2\;\! ...\;\! t_k}\Big)\in (0\;\! ,+\infty )\, ,$
\item[(4)] $\displaystyle (0\;\! ,+\infty )\ni t\longmapsto \ln |\omega (t)|\in (0\;\! ,+\infty )\, ,$
where $\omega\in\mathbf{\Omega}$ is defined by $\displaystyle \omega (z)=\prod\limits_{j=1}^{\infty}
\Big( 1+\frac{iz}{t_j}\Big)\;\! ,z\in\mathbb{C}\;\! ,$
\end{itemize}
satisfy the non-quasianalyticity condition.}
\smallskip

Since  $\displaystyle \int\limits_{t_1}^{t_{k+1}}\! \frac{n(t)}{t^2}\;\!{\rm d}t=
\sum\limits_{j=1}^{k}\int\limits_{t_j}^{t_{j+1}}\! \frac{n(t)}{t^2}\;\!{\rm d}t =
\sum\limits_{j=1}^{k} j\;\! \Big( \frac 1{t_j}-\frac 1{t_{j+1}}\Big)=
\Big(\sum\limits_{j=1}^{k}\frac 1{f_j}\Big) -\frac k{t_{k+1}}\;\!$,
we have $\displaystyle \int\limits_{t_1}^{\infty} \frac{n(t)}{t^2}\;\!{\rm d}t\leq \sum\limits_{j=1}^{\infty}
\frac 1{f_j}<+\infty\;\!$.
\smallskip

A proof of the non-quasianalyticity of $\ln |\omega (\;\!\cdot\;\! )|$ can be found, for example,
in the proof of implication $({\rm iii})\Rightarrow ({\rm i})$ in \cite{C-Z1}, Theorem 1.6.

The  non-quasianalyticity of $N(\;\!\cdot\;\! )$ follows from the non-quasianalyticity of
$\ln |\omega (\;\!\cdot\;\! )|$ and the clear inequality $N(t)\leq \ln |\omega (t)|\;\! ,t>0\;\!$.

Finally, the non-quasianalyticity of $\alpha$ is consequence of the inequality
\medskip

\noindent\hspace{1.67 cm}$\displaystyle 1+\sum\limits_{k=1}^{\infty}\frac{t^k}{t_1\;\! ...\;\! t_k} =
1+\sum\limits_{k=1}^{\infty}\frac 1{2^k}\;\! \frac{t^k}{(t_1/2)\;\! ...\;\! (t_k/2)}$

\noindent\hspace{4.16 cm}$\displaystyle \leq \Big(\sum\limits_{k=1}^{\infty}\frac 1{2^k}\Big)
\max \Big( 1\;\! ,\sup\limits_{k\geq 1}\frac{t^k}{(t_1/2)\;\! ...\;\! (t_k/2)}\Big)$
\medskip

\noindent and of the non-quasianalyticity of $N(\;\!\cdot\;\! )$ with $t_k$ replaced by $t_k/2\;\!$.
\end{remark}

The goal of this section is to show, how we can majorize functions of a certain regularity,
satisfying the non-quasianalyticity condition, with more regular or more explicite functions,
still satisfying the non-quasianalyticity condition.
We consider three function groups$\;\! :$
\begin{itemize}
\item Increasing functions $(0\;\! ,+\infty )\longrightarrow (0\;\! ,+\infty )\;\!$.
\item "Concave like functions" $\alpha : (0\;\! ,+\infty )\longrightarrow (0\;\! ,+\infty )\;\!$, which can be
\smallskip

\begin{itemize}
\item[1)] concave: $\alpha\big( (1-\lambda )t_1+\lambda t_2\big)\geq (1-\lambda )\alpha (t_1)+\lambda
\alpha (t_2)$ for $0\leq\lambda\leq 1$ and $t_1\;\! t_2>0\;\!$;
\item[2)] such that $\displaystyle (0\;\! ,+\infty )\ni t\longmapsto\frac{\alpha (t)}t$ is decreasing;
\item[3)] subadditive: $\alpha (t_1+t_2)\leq\alpha (t_1)+\alpha (t_2)$ for $t_1\;\! ,t_2>0\;\!$.
\end{itemize}
\smallskip

\noindent We notice that $1)\Rightarrow 2)\Rightarrow 3)\;\!$.
Indeed, if $\alpha$ is concave and $0<t_1<t_2$ are arbitrary, then we have for any $0<\varepsilon <t_1$
\begin{equation*}
\begin{split}
\hspace{8 mm}\alpha (t_1)=\;&\alpha\Big( \frac{t_2-t_1}{t_2-\varepsilon}\;\! \varepsilon +
\frac{t_1-\varepsilon}{t_2-\varepsilon}\;\! t_2\Big)\geq \frac{t_2-t_1}{t_2-\varepsilon}\;\! \alpha (\varepsilon )+
\frac{t_1-\varepsilon}{t_2-\varepsilon}\;\!\alpha (t_2) \\
>\;&\frac{t_1-\varepsilon}{t_2-\varepsilon}\;\!\alpha (t_2)\;\! .
\end{split}
\end{equation*}
Letting $\varepsilon\to 0$ we conclude that $\displaystyle \alpha (t_1)\geq\frac{t_1}{t_2}\;\!\alpha (t_2)
\Longleftrightarrow \frac{\alpha (t_1)}{t_1}\geq \frac{\alpha (t_2)}{t_2}\;\!$.
On the other hand, if $\displaystyle (0\;\! ,+\infty )\ni t\longmapsto\frac{\alpha (t)}t$ is decreasing, then we
have for all $t_1\;\! ,t_2>0\;\!$:
\begin{equation*}
\begin{split}
\alpha (t_1+t_2)=\;&t_1\;\!\frac{\alpha (t_1+t_2)}{t_1+t_2}+t_2\;\!\frac{\alpha (t_1+t_2)}{t_1+t_2} \\
\leq\;&t_1\;\!\frac{\alpha (t_1)}{t_1}+t_2\;\!\frac{\alpha (t_2)}{t_2} =\alpha (t_1)+\alpha (t_2)\;\! .
\end{split}
\end{equation*}
\item $(0\;\! ,+\infty )\ni t\longmapsto\ln |\omega (t)|$ with $\omega$ an entire function belonging to
$\mathbf{\Omega}$ or $\mathbf{\Omega}_0\;\!$.
\end{itemize}

The next lemma extends \cite{C-Z1}, Lemma 1.7$\;\! :$

\begin{lemma}\label{concave-explicite}
If
\begin{equation*}
0<t_1\leq \frac{t_2}2\leq \frac{t_3}3\leq\;\! ...\;\! \leq +\infty\;\! ,\quad t_1<+\infty\;\! ,
\end{equation*}
then the function
\begin{equation*}
\alpha : (0\;\! ,+\infty )\ni t\longmapsto \ln\Big(1+\sum\limits_{k=1}^{\infty}\frac{t^k}{t_1\;\! ...\;\! t_k}\Big)
\in (0\;\! ,+\infty )
\end{equation*}
is strictly increasing and concave. Assuming additionally that $\displaystyle \frac{t_k}k\neq\frac{t_{k+1}}{k+1}$
for at least one $k\geq 1\;\!$, $\alpha$ turns out to be even strictly concave.
\end{lemma}

\begin{proof}
$\alpha$ is clearly strictly increasing.

For the proof of the concavity it is convenient to denote $\displaystyle c_k=\frac{t_k}k\;\! ,k\geq 1\;\! .$
Then $c_1\geq c_2\geq c_3\geq\;\! ...\;\!\geq 0\;\! ,c_1>0$ and
\smallskip

\centerline{$\displaystyle \alpha (t)=\ln\bigg( \sum\limits_{k=0}^{\infty} \Big(\prod\limits_{j=1}^kc_j\Big)\;\!
\frac{t^k}{k!}\bigg)\;\! \qquad t>0\;\! ,$}

\noindent where we agree that $\displaystyle \prod\limits_{j=1}^kc_j=1$ for $k=0\;\!$.

If $\displaystyle c_k=\frac{t_k}k=\frac{t_{k+1}}{k+1}=c_{k+1}$ for all $k\geq 1\;\!$, then
$\alpha (t)=\ln e^{c_1t}=c_1 t\;\!$, so $\alpha$ is linear, hence concave.
We will show that, assuming $c_{k}>c_{k+1}$ for some $k\geq 1\;\!$,
$\alpha$ is strictly concave, by proving that $\alpha''(t)<0$ for all $t>0\;\!$.
Let $k_0$ denote the least integer $k\geq 1$ for which $c_{k}>c_{k+1}\;\!$.

Denoting $\displaystyle f(t)=\sum\limits_{k=0}^{\infty} \Big(\prod\limits_{j=1}^kc_j\Big)\;\! \frac{t^k}{k!}\;\!$,
we have
\begin{equation*}
\alpha''(t)=\big(\ln f(t)\big)''=\Big(\frac{f'(t)}{f(t)}\Big)'=\frac{f''(t) f(t)-f'(t)^2}{f(t)^2}\;\!.
\end{equation*}
Therefore out task is to prove that $f(t)^2-f''(t) f(t)>0$ for all $t>0\;\!$.

Computation yields $\displaystyle f'(t)=\sum\limits_{k=0}^{\infty}
\Big(\prod\limits_{j=1}^{k+1}c_j\Big)\;\! \frac{t^k}{k!}\;\!$,
$\displaystyle f''(t)=\sum\limits_{k=0}^{\infty} \Big(\prod\limits_{j=1}^{k+2}c_j\Big)\;\! \frac{t^k}{k!}$
and
\begin{equation*}
\begin{split}
&f(t)^2-f''(t) f(t) \\
=\;&\sum\limits_{k=0}^{\infty}\bigg(
\sum\limits_{\substack{p,q\geq 0 \\ p+q=k}}\frac 1{p! q!}
\Big(\prod\limits_{j=1}^{p+1}c_j\Big) \Big(\prod\limits_{j=1}^{q+1}c_j\Big)
-\!\!\sum\limits_{\substack{p,q\geq 0 \\ p+q=k}}\frac 1{p! q!}
\Big(\prod\limits_{j=1}^{p+2}c_j\Big)\Big( \prod\limits_{j=1}^qc_j\Big)\!\bigg) t^k.
\end{split}
\end{equation*}
Hence the proof will be done once we show that
\begin{equation*}
C_k:=\sum\limits_{\substack{p,q\geq 0 \\ p+q=k}}\frac 1{p! q!}
\Big(\prod\limits_{j=1}^{p+1}c_j\Big) \Big(\prod\limits_{j=1}^{q+1}c_j\Big)
-\!\!\sum\limits_{\substack{p,q\geq 0 \\ p+q=k}}\frac 1{p! q!}
\Big(\prod\limits_{j=1}^{p+2}c_j\Big)\Big( \prod\limits_{j=1}^qc_j\Big)
\geq 0
\end{equation*}
for all $k\geq 0\;\!$, and $C_{k_0-1}> 0\;\!$.
\smallskip

Since $C_0=c_1 c_1-c_1 c_2=c_1(c_1-c_2)\geq 0\;\!$, where the inequality is strict if
$k_0=1\;\!$, it remains that we prove that $C_k\geq 0$ for all $k\geq 1$ and
$C_{k_0-1}>0$ if $k_0\geq 2\;\!$..
\smallskip

For each $k\geq 1\;\!$, using
\medskip

\noindent\hspace{3 mm}$\displaystyle
\sum\limits_{\substack{p,q\geq 0 \\ p+q=k}}\frac 1{p! q!}
\Big(\prod\limits_{j=1}^{p+1}c_j\Big) \Big(\prod\limits_{j=1}^{q+1}c_j\Big)
=\frac 1{k!}\;\! c_1 \prod\limits_{j=1}^{k+1}c_j +
\sum\limits_{\substack{p\geq 1,q\geq 0 \\ p+q=k}}\frac 1{p! q!}
\Big(\prod\limits_{j=1}^{p+1}c_j\Big) \Big(\prod\limits_{j=1}^{q+1}c_j\Big)$
\medskip

\noindent and

\noindent\hspace{7.1 mm}$\displaystyle
\sum\limits_{\substack{p,q\geq 0 \\ p+q=k}}\frac 1{p! q!}
\Big(\prod\limits_{j=1}^{p+2}c_j\Big)\Big( \prod\limits_{j=1}^qc_j\Big)
=\frac 1{k!}\;\! \prod\limits_{j=1}^{k+2}c_j +
\sum\limits_{\substack{p\geq 0,q\geq 1 \\ p+q=k}}\frac 1{p! q!}
\Big(\prod\limits_{j=1}^{p+2}c_j\Big)\Big( \prod\limits_{j=1}^qc_j\Big)$

\noindent\hspace{2.3 mm}$\displaystyle =
\frac 1{k!}\;\! \prod\limits_{j=1}^{k+2}c_j +
\sum\limits_{\substack{p'\geq 1,q'\geq 0 \\ p'+q'=k}}\frac 1{(p'-1)! (q'+1)!}
\Big(\prod\limits_{j=1}^{p'+1}c_j\Big)\Big( \prod\limits_{j=1}^{q'+1}c_j\Big)\;\! ,$
\smallskip

\noindent we obtain
\begin{equation}\label{C_k}
C_k=\frac 1{k!}\;\! \big( c_1-c_{k+2}\big) \prod\limits_{j=1}^{k+1}c_j +S_k\;\! ,
\end{equation}
where
\begin{equation}\label{initial-S_k}
S_k=\!\sum\limits_{\substack{p\geq 1,q\geq 0 \\ p+q=k}}\!
\Big(\frac 1{p! q!}-\frac 1{(p-1)! (q+1)!}\Big)
\Big(\prod\limits_{j=1}^{p+1}c_j\Big) \Big(\prod\limits_{j=1}^{q+1}c_j\Big)\;\! .
\end{equation}

\noindent Therefore it is enough to show that $S_k\geq 0$ for all $k\geq 1\;\!$. Indeed,
then (\ref{C_k}) yields $C_k\geq 0$ for all $k\geq 1\;\!$.
Moreover, if $k_0\geq 2$ and so $k_0-1\geq 1\;\!$, then

\noindent (\ref{C_k}) and $c_1-c_{k_0+1}\geq c_{k_0}-c_{k_0+1}>0$ yield also
\smallskip

\centerline{$\displaystyle C_{k_0-1}=\frac 1{(k_0-1)!}\;\! \big( c_1-c_{k_0+1}\big)
\prod\limits_{j=1}^{k_0}c_j +S_{k_0-1}> S_{k_0-1}\geq 0\;\! .$}
\smallskip

Direct computation shows that $S_k\geq 0$ for $1\leq k\leq 5\;\!$:
\smallskip

\noindent\hspace{9 mm}$\displaystyle S_1=0\;\! ,\quad S_2=\frac 12\;\! c_1^2c_2(c_1-c_3)
\geq 0\;\! ,\quad S_3=\frac 23\;\! c_1^2c_2 c_3(c_2-c_4)\geq 0\;\! ,$

\noindent\hspace{9 mm}$\displaystyle S_4=\frac 18\;\! c_1^2c_2 c_3 c_4 (c_2-c_5)+
\frac 1{12}\;\! c_1^2c_2^2c_3 (c_3-c_4)\geq 0\;\! ,$

\noindent\hspace{9 mm}$\displaystyle S_5=\frac 1{30}\;\! c_1^2c_2 c_3 c_4 c_5 (c_2-c_6) +
\frac 1{24}\;\! c_1^2c_2^2c_3c_4 (c_3-c_5)\geq 0\;\! .$
\smallskip

\noindent It remains to show that $S_k\geq 0$ for all $k\geq 6\;\!$.

Let in the sequel the integer $k\geq 6$ be arbitrary.
For $\displaystyle p=\frac{k+1}2$ (what can happen only for odd $k$) we have
\smallskip

\centerline{$\displaystyle \frac 1{p! q!}-\frac 1{(p-1)! (q+1)!}=0\;\! ,$}

\noindent hence
\smallskip

\noindent\hspace{1.16 cm}$\displaystyle
S_k=\sum\limits_{\substack{1\leq p\leq k/2 \\ p+q=k}}\!
\Big(\frac 1{p! q!}-\frac 1{(p-1)! (q+1)!}\Big)
\Big(\prod\limits_{j=1}^{p+1}c_j\Big) \Big(\prod\limits_{j=1}^{q+1}c_j\Big)$

\noindent\hspace{2.19 cm}$\displaystyle
+\sum\limits_{\substack{k/2 +1\leq p\leq k \\ p+q=k}}\!
\Big(\frac 1{p! q!}-\frac 1{(p-1)! (q+1)!}\Big)
\Big(\prod\limits_{j=1}^{p+1}c_j\Big) \Big(\prod\limits_{j=1}^{q+1}c_j\Big)$

\noindent\hspace{1.7 cm}$\displaystyle
=\sum\limits_{\substack{1\leq p\leq k/2 \\ p+q=k}}\!
\Big(\frac 1{p! q!}-\frac 1{(p-1)! (q+1)!}\Big)
\Big(\prod\limits_{j=1}^{p+1}c_j\Big) \Big(\prod\limits_{j=1}^{q+1}c_j\Big)$

\noindent\hspace{2.19 cm}$\displaystyle
+\sum\limits_{\substack{0\leq q\leq k/2 -1 \\ p+q=k}}\!
\Big(\frac 1{p! q!}-\frac 1{(p-1)! (q+1)!}\Big)
\Big(\prod\limits_{j=1}^{p+1}c_j\Big) \Big(\prod\limits_{j=1}^{q+1}c_j\Big)$

\noindent\hspace{1.7 cm}$\displaystyle
=\sum\limits_{\substack{1\leq p\leq k/2 \\ p+q=k}}\!
\Big(\frac 1{p! q!}-\frac 1{(p-1)! (q+1)!}\Big)
\Big(\prod\limits_{j=1}^{p+1}c_j\Big) \Big(\prod\limits_{j=1}^{q+1}c_j\Big)$

\noindent\hspace{2.19 cm}$\displaystyle
+\sum\limits_{\substack{0\leq p\leq k/2 -1 \\ p+q=k}}\!
\Big(\frac 1{q! p!}-\frac 1{(q-1)! (p+1)!}\Big)
\Big(\prod\limits_{j=1}^{q+1}c_j\Big) \Big(\prod\limits_{j=1}^{p+1}c_j\Big);\! .$

Denoting by $p_0$ the unique integer for which $\displaystyle \frac{k-1}2\leq p_0\leq \frac k2\;\!$,
it follows$\;\! :$

\noindent\hspace{0.8 mm}$\displaystyle
S_k=\sum\limits_{\substack{1\leq p\leq p_0-1 \\ p+q=k}}\!
\Big(\frac 1{p! q!}-\frac 1{(p-1)! (q+1)!}\Big)
\Big(\prod\limits_{j=1}^{p+1}c_j\Big) \Big(\prod\limits_{j=1}^{q+1}c_j\Big)$

\noindent\hspace{6.8 mm}$\displaystyle
+\;\Big(\frac 1{p_0! (k-p_0)!}-\frac 1{(p_0-1)! (k-p_0+1)!}\Big)
\Big(\prod\limits_{j=1}^{p_0+1}c_j\Big) \Big(\prod\limits_{j=1}^{k-p_0+1}c_j\Big)$

\noindent\hspace{6.8 mm}$\displaystyle
+\sum\limits_{\substack{1\leq p\leq p_0 -1 \\ p+q=k}}\!
\Big(\frac 1{p! q!}-\frac 1{(p+1)! (q-1)!}\Big)
\Big(\prod\limits_{j=1}^{p+1}c_j\Big) \Big(\prod\limits_{j=1}^{q+1}c_j\Big)$

\noindent\hspace{6.8 mm}$\displaystyle
+\;\Big(\frac 1{k!}-\frac 1{(k-1)!}\Big)\;\! c_1 \Big(\prod\limits_{j=1}^{k+1}c_j\Big)$

\noindent\hspace{6 mm}$\displaystyle
=\Big(\frac 1{p_0! (k-p_0)!}-\frac 1{(p_0-1)! (k-p_0+1)!}\Big)
\Big(\prod\limits_{j=1}^{p_0+1}c_j\Big) \Big(\prod\limits_{j=1}^{k-p_0+1}c_j\Big)$

\noindent\hspace{6.8 mm}$\displaystyle
+\;\Big(\frac 1{k!}-\frac 1{(k-1)!}\Big)\;\! c_1 \Big(\prod\limits_{j=1}^{k+1}c_j\Big)$

\noindent\hspace{6.8 mm}$\displaystyle
+\sum\limits_{\substack{1\leq p\leq p_0 -1 \\ p+q=k}}\!
\Big(\frac 2{p! q!}-\frac 1{(p-1)! (q+1)!}-\frac 1{(p+1)! (q-1)!}\Big)
\Big(\prod\limits_{j=1}^{p+1}c_j\Big) \Big(\prod\limits_{j=1}^{q+1}c_j\Big) .$
\smallskip

Set
\smallskip

\centerline{$\displaystyle
d_{k,p_0}:=\frac 1{p_0! (k-p_0)!}-\frac 1{(p_0-1)! (k-p_0+1)!}=
\frac{k-2\;\! p_0+1}{\;\! p_0! (k-p_0+1)!\;\!}>0$}
\medskip

\noindent and, for $1\leq p\leq p_0 -1\;\!$,
\medskip

\centerline{$\displaystyle
d_{k,p}:=\frac 2{p! (k-p)!}-\frac 1{(p-1)! (k-p+1)!}-\frac 1{(p+1)! (k-p-1)!}\, .$}
\smallskip

\noindent Then
\begin{equation}\label{final-S_k}
\begin{split}
S_k=\;&d_{k,p_0}
\Big(\prod\limits_{j=1}^{p_0+1}c_j\Big) \Big(\prod\limits_{j=1}^{k-p_0+1}c_j\Big)
-\Big(\frac 1{(k-1)!}-\frac 1{k!}\Big)\;\! c_1 \Big(\prod\limits_{j=1}^{k+1}c_j\Big) \\
&+\sum\limits_{1\leq p\leq p_0 -1} d_{k,p}
\Big(\prod\limits_{j=1}^{p+1}c_j\Big) \Big(\prod\limits_{j=1}^{k-p+1}c_j\Big) .
\end{split}
\end{equation}

It is easy to see that
\begin{equation}\label{sign-d_p}
d_{k,p}\geq 0\text{ if }p\geq\frac{k-\sqrt{k+2}}2\;\! ,\qquad
d_{k,p}< 0\text{ if }p<\frac{k-\sqrt{k+2}}2\;\! .
\end{equation}
Let $p_1$ denote the unique integer for which
\medskip

\centerline{$\displaystyle \frac{k-\sqrt{k+2}}2\leq p_1<\frac{k-\sqrt{k+2}}2+1\;\! .$}
\medskip

\noindent If $k=6\;\!$, then $p_0=3$ and $p_1=2\;\!$, while if $k\geq 7\;\!$, then
\medskip

\centerline{$\displaystyle 2\leq \frac{k-\sqrt{k+2}}2\leq\frac{k-1}2\leq
p_1<\frac{k-\sqrt{k+2}}2\leq\frac{k-1}2+1\leq p_0\;\! .$}
\medskip

\noindent Thus we always have $2\leq p_1\leq p_0-1\;\!$.

Since the function
\smallskip

\centerline{$\displaystyle \{ 0\;\! ,1\;\! ,2\;\! ,\;\! ...\;\! ,p_0\}\ni p\longmapsto
\Big(\prod\limits_{j=1}^{p+1}c_j\Big) \Big(\prod\limits_{j=1}^{k-p+1}c_j\Big)
=\Big(\prod\limits_{j=1}^{p+1}c_j\Big)^{\! 2} \Big(\prod\limits_{j=p+2}^{k-p+1}c_j\Big)$}
\smallskip

\noindent is increasing and, according to (\ref{sign-d_p}),
\begin{equation*}
d_{k,p}\geq 0\text{ if }p\geq p_1\;\! ,\quad d_{k,p}< 0\text{ if }p\leq p_1-1\;\! ,
\end{equation*}
we deduce
\begin{equation*}
d_{k,p} \Big(\prod\limits_{j=1}^{p+1}c_j\Big) \Big(\prod\limits_{j=1}^{k-p+1}c_j\Big)\geq
d_{k,p} \Big(\prod\limits_{j=1}^{p_1+1}c_j\Big) \Big(\prod\limits_{j=1}^{k-p_1+1}c_j\Big)\;\! ,
\qquad 1\leq p\leq p_0-1\;\! .
\end{equation*}
We have also
\smallskip

\noindent\hspace{2.42 cm}$\displaystyle
\Big(\prod\limits_{j=1}^{p_0+1}c_j\Big) \Big(\prod\limits_{j=1}^{k-p_0+1}c_j\Big)\geq
\Big(\prod\limits_{j=1}^{p_1+1}c_j\Big) \Big(\prod\limits_{j=1}^{k-p_1+1}c_j\Big)\;\! ,$

\noindent\hspace{4.2 cm}$\displaystyle
c_1 \Big(\prod\limits_{j=1}^{k+1}c_j\Big)\leq
\Big(\prod\limits_{j=1}^{p_1+1}c_j\Big) \Big(\prod\limits_{j=1}^{k-p_1+1}c_j\Big)\;\! ,$
\smallskip

\noindent so (\ref{final-S_k}) yields
\begin{equation}\label{minorized}
S_k\geq \bigg( d_{k,p_0}-\Big(\frac 1{(k-1)!}-\frac 1{k!}\Big) +\sum\limits_{1\leq p\leq p_0 -1} d_{k,p}
\bigg) \Big(\prod\limits_{j=1}^{p_1+1}c_j\Big) \Big(\prod\limits_{j=1}^{k-p_1+1}c_j\Big)
\end{equation}

In order to compute the sum
\medskip

\centerline{$\displaystyle s_k:=d_{k,p_0}-\Big(\frac 1{(k-1)!}-\frac 1{k!}\Big) +
\sum\limits_{1\leq p\leq p_0 -1} d_{k,p}\;\! ,$}
\smallskip

\noindent we notice that, according to (\ref{final-S_k}), $s_k$ is equal to $S_k$ with
$c_1=c_2=\;\! ...\;\!$. Computing $S_k$ in this case by using the formula (\ref{initial-S_k})
instead of (\ref{final-S_k}), we obtain
\smallskip

\centerline{$\displaystyle s_k=
\sum\limits_{p=1}^k\Big(\frac 1{p! (k-p)!}-\frac 1{(p-1)! (k-p+1)!}\Big)\;\! .$}
\medskip

\noindent But this is a telescoping sum, hence it is equal to
$\displaystyle \frac 1{k! 0!}-\frac 1{0! k!}=0\;\! .$
\smallskip

Using now (\ref{minorized}), we deduce the desired result: $S_k\geq 0\;\!$.

\end{proof}

The next majorization theorem is essentially \cite{I-M}, Theorem 1 and  \cite{C-Z1},

\noindent Theorem 1.6, claiming that any increasing function, which satisfies the non-quasianalyticity
condition, can be majorized by some function $c+\ln |\omega (\;\!\cdot\;\! )|$ with
$\omega\in\mathbf{\Omega}$ and $c\geq 0$ a constant$\;\! :$

\begin{theorem}\label{increasing}
For $f : (0\;\! ,+\infty )\longrightarrow (0\;\! ,+\infty )$ the following conditons are

\noindent equivalent$\;\! :$
\begin{itemize}
\item[(i)] $f(t)\leq\alpha (t)\;\! ,t>0\;\! ,$ for $\alpha : (0\;\! ,+\infty )\longrightarrow (0\;\! ,+\infty )$
some increasing

\noindent function satisfying the non-quasianalyticity condition.
\item[(ii)] There exist

\centerline{$\displaystyle 0<t_1\leq t_2\leq t_3\leq\;\! ...\;\! \leq +\infty\;\! ,\quad t_1<+\infty\;\! ,\quad
\sum\limits_{j=1}^{\infty}\frac 1{t_j}<+\infty$}

\noindent and a constant $c\geq 0\;\!$, such that
\begin{equation*}
f(t)\leq c+\ln \max \Big( 1\;\! ,\sup\limits_{k\geq 1}\frac{t^k}{t_1\;\! t_2\;\! ...\;\! t_k}\Big)\;\! ,\qquad t>0\;\! .
\end{equation*}
\item[(iii)] $f(t)\leq c+\ln |\omega (t)|\;\! ,t>0\;\! ,$ for some $\omega\in\mathbf{\Omega}$ and
constant $c\geq 0\;\!$.
\end{itemize}
A necessary condition that $f$ satisfies the above equivalent conditions is
\begin{equation}\label{nec.incr}
\lim\limits_{t\to +\infty}\frac{f(t)}t=0\;\! .
\end{equation}
\end{theorem}

\begin{proof}
The equivalences $({\rm i})\Leftrightarrow ({\rm ii})\Leftrightarrow ({\rm iii})$ are immediate
consequences of the corresponding equivalences in \cite{C-Z1}, Theorem 1.6.
\smallskip

Also the necessary condition (\ref{nec.incr}) is well-known. Here is a short proof of it$\;\! :$

Let $\alpha$ be as in (i). Then
\smallskip

\centerline{$\displaystyle
0\leq\frac{f(t)}t\leq \frac{\alpha (t)}t=\int\limits_t^{+\infty}\!\frac{\alpha (t)}{s^2}\;\!{\rm d}s\leq
\int\limits_t^{+\infty}\!\frac{\alpha (s)}{s^2}\;\!{\rm d}s\xrightarrow{\; t\to +\infty\;}0\; .$}

\end{proof}

The second majorization theorem is an extended version of \cite{I-M}, Theorem 2 and 
\cite{C-Z1}, Theorem 1.8. It claims essentially that Lebesgue measurable positive subadditive
functions on $(0\;\! ,+\infty )\;\!$, which are bounded on $(0\;\! ,1]$ and satisfy the
non-quasianalyticity condition, can be majorized by a continuous, increasing, concave function
satisfying the non-quasianalyticity condition, or by a function of the form $c+\ln |\omega (\;\!\cdot\;\! )|$
with $\omega\in\mathbf{\Omega}_0$ and $c\geq 0$ a constant$\;\! :$

\begin{theorem}\label{concave}
For $f : (0\;\! ,+\infty )\longrightarrow (0\;\! ,+\infty )$ the following conditons are

\noindent equivalent$\;\! :$
\begin{itemize}
\item[(i)] $f(t)\leq\alpha (t)\;\! ,t>0\;\! ,$ for $\alpha : (0\;\! ,+\infty )\longrightarrow (0\;\! ,+\infty )$
some Lebesgue

\noindent measurable, subadditive function, bounded on $(0\;\! ,1]$ and satisfying the
non-quasianalyticity condition.
\item[(ii)] $f(t)\leq\alpha (t)\;\! ,t>0\;\! ,$ for $\alpha : (0\;\! ,+\infty )\longrightarrow (0\;\! ,+\infty )$
some continuous

\noindent function, bounded on $(0\;\! ,1]$ and satisfying the non-quasianalyticity

\noindent condition, such that $\displaystyle (0\;\! ,+\infty )\ni t\longmapsto\frac{\alpha (t)}t$ is decreasing.
\item[(iii)] $f(t)\leq\alpha (t)\;\! ,t>0\;\! ,$ with $\alpha : (0\;\! ,+\infty )\longrightarrow (0\;\! ,+\infty )$
some increasing, concave function satisfying the non-quasianalyticity condition.
\item[(iv)] $f(t)\leq\alpha (t)\;\! ,t>0\;\! ,$ with $\alpha : (0\;\! ,+\infty )\longrightarrow (0\;\! ,+\infty )$
some infinitely

\noindent differentiable, increasing, concave function satisfying the non-

\noindent quasianalyticity condition.
\item[(v)] There exist

\centerline{$\displaystyle 0<t_1\leq \frac{t_2}2\leq \frac{t_3}3\leq\;\! ...\;\! \leq +\infty\;\! ,\quad t_1<+\infty
\;\! ,\quad \sum\limits_{j=1}^{\infty}\frac 1{t_j}<+\infty$}

\noindent and a constant $c\geq 0\;\!$, such that
\begin{equation*}
f(t)\leq c+\ln\Big(1+\sum\limits_{k=1}^{\infty}\frac{t^k}{t_1\;\! ...\;\! t_k}\Big)\;\! ,\qquad t>0\;\! .
\end{equation*}
\item[(vi)] There exist

\centerline{$\displaystyle 0<t_1\leq \frac{t_2}2\leq \frac{t_3}3\leq\;\! ...\;\! \leq +\infty\;\! ,\quad t_1<+\infty
\;\! ,\quad \sum\limits_{j=1}^{\infty}\frac 1{t_j}<+\infty$}

\noindent and a constant $c\geq 0\;\!$, such that
\begin{equation*}
f(t)\leq c+\ln \max \Big( 1\;\! ,\sup\limits_{k\geq 1}\frac{t^k}{t_1\;\! t_2\;\! ...\;\! t_k}\Big)\;\! ,\qquad t>0\;\! .
\end{equation*}
\item[(vii)] $f(t)\leq c+\ln |\omega (t)|\;\! ,t>0\;\! ,$ for some $\omega\in\mathbf{\Omega}_0$ and
constant $c\geq 0\;\!$.
\end{itemize}
A necessary condition that $f$ satisfies the above equivalent conditions is
\begin{equation}\label{nec.subadd}
\lim\limits_{t\to +\infty}\frac{\;\! f(t) \ln t}t=0\;\! .
\end{equation}
\end{theorem}

\begin{proof}
The equivalences $({\rm i})\Leftrightarrow ({\rm vi})\Leftrightarrow ({\rm vii})$ are immediate
consequences of the equivalences $({\rm iii})\Leftrightarrow ({\rm ii})\Leftrightarrow ({\rm iv})$
in \cite{C-Z1}, Theorem 1.8.

Implications $({\rm vi})\Rightarrow ({\rm v})$ and $({\rm iv})\Rightarrow ({\rm iii})$ are trivial,
while implication

\noindent $({\rm v})\Rightarrow ({\rm iv})$ follows by Lemma \ref{concave-explicite}
and Remark \ref{explicite}.

Finally, the implications $({\rm iii})\Rightarrow ({\rm ii})\Rightarrow ({\rm i})$ where proved in
the discussion before Lemma \ref{concave-explicite}.

The necessary condition (\ref{nec.subadd}) is, like (\ref{nec.incr}) in Theorem \ref{increasing},
well-known. We provide a short proof of it, essentially reproducing the proof of \cite{Bj}, Corollary 1.2.8$\;\! :$

Let $\alpha$ be as in (ii). We have for every $t>1$
\begin{equation*}
\int\limits_{\sqrt{t}}^{+\infty}\frac{\alpha (s)}{s^2}\;\!{\rm d}s\geq
\int\limits_{\sqrt{t}}^t\frac{\alpha (s)}{s^2}\;\!{\rm d}s\geq
\int\limits_{\sqrt{t}}^t\frac{\alpha (t)}t\;\!\frac 1{s}\;\!{\rm d}s=\frac{\alpha (t)}t\;\! \ln\frac t{\sqrt{t}}=
\frac 12\;\! \frac{\alpha (t)\ln t}t\;\! ,
\end{equation*}
so
\begin{equation*}
0\leq \frac{\alpha (t)\ln t}t\leq 2 \int\limits_{\sqrt{t}}^{+\infty}\frac{\alpha (s)}{s^2}\;\!{\rm d}s
\xrightarrow{\; t\to +\infty\;}0\; .
\end{equation*}
\end{proof}

We notice that $({\rm i})\Leftrightarrow ({\rm iii})$ in Theorem \ref{concave} was originally
proved by A. Beurling (see \cite{Be1}, lemma 1, \cite{Bj}, Theorem 1.2.7, \cite{Be2}, Lemma V),
\cite{G}, Lemma 3.3). A new feature of Theorem \ref{concave} consists in the exhibition
(thanks to Lemma \ref{concave-explicite}) of a rather explicite $\alpha$ in (iii), obtaining thus
the equivalent conditions (iv) and (v).

Clearly, every $f$, which satisfies the equivalent conditions in Theorem \ref{concave}, satisfies also
the equivalent conditions in Theorem \ref{increasing}.
It is an intriguing question: does it exist $f$ satisfying the conditions in Theorem \ref{increasing},
but not those in Theorem \ref{concave} ? The answer is yes$\;\! :$

\begin{corollary}
There exists an increasing function $(0\;\! ,+\infty )\longrightarrow (0\;\! ,+\infty )$ satisfying the
non-quasianalyticity condition, which can not be majorized by

\noindent any Lebesgue measurable, subadditive function on $(0\;\! ,+\infty )\;\!$, which is bounded
on $(0\;\! ,1]$ and satisfies the non-quasianalyticity condition.

Consequently there exists $\omega\in\mathbf{\Omega}$ such that $|\omega (\;\!\cdot\;\! )|$ can not
be majorized by a scalar multiple of some $|\rho (\;\!\cdot\;\! )|$ with $\rho\in\mathbf{\Omega}_0\;\!$.
\end{corollary}

\begin{proof}
Let $e= t_1<t_2<t_3<\;\! ...$ be a sequence such that $\displaystyle \sum\limits_{k=1}^{\infty}
\frac 1{\ln t_k}<+\infty$ (for example, $t_k=e^{k^2}$). 
Defining the function $f : (0\;\! ,+\infty )\longrightarrow (0\;\! ,+\infty )$ by
\medskip

\noindent\hspace{3.2 cm}$\displaystyle f(t):=0\text{ for }0<t<t_1\;\! ,$

\noindent\hspace{3.2 cm}$\displaystyle f(t):=\frac{t_k}{\ln t_k}\text{ for }t_k\leq t< t_{k+1}\;\! ,k\geq 1\;\! ,$
\medskip

\noindent $f$ will be increasing and satisfying the non-quasianalyticity condition$\;\! :$
\begin{equation*}
\int\limits_1^{+\infty}\!\frac{f(t)}{t^2}\;\!{\rm d}t=\sum\limits_{k=1}^{\infty}
\int\limits_{t_k}^{t_{k+1}}\!\frac{f(t)}{t^2}\;\!{\rm d}t=\sum\limits_{k=1}^{\infty}\;\!
\frac{t_k}{\ln t_k} \Big(\frac 1{t_k}-\frac 1{t_{k+1}}\Big)\leq\sum\limits_{k=1}^{\infty}\;\!
\frac 1{\ln t_k}<+\infty\;\! .
\end{equation*}

The above defined $f$ can not be majorized by any Lebesgue measurable, subadditive function
on $(0\;\! ,+\infty )\;\!$, which is bounded on $(0\;\! ,1]$ and satisfies the non-quasianalyticity condition.
Indeed, otherwise (\ref{nec.subadd}) would hold true byTheorem \ref{concave}, contradicting
$\displaystyle \frac{\;\! f(t_k) \ln t_k}{t_k}=1\;\! ,k\geq 1\;\!$.

\end{proof}

\section{The mild strong non-quasianalyticity condition}

First at all we notice that if $\alpha\;\! ,\beta : (0\;\! ,+\infty )\longrightarrow (0\;\! ,+\infty )$ are increasing,
$\displaystyle \int\limits_1^{+\infty}\!\frac{\alpha (t)}{t^2}\;\!{\rm d}t<+\infty$ and $\displaystyle
\lim\limits_{t\to +\infty}\frac{\beta (t)}t=0\;\!$, then $\displaystyle \int\limits_1^{+\infty}\!\frac{\alpha (t)}{t^2}
\ln \frac t{\beta (t)}\;\!{\rm d}t> -\infty$ is a well defined improper integral.
Indeed, if $t_0\geq 1$ is such that $\displaystyle \frac{\beta (t)}t<1$ for $t\geq t_0\;\!$, then
$\displaystyle [\;\! t_0\;\! ,+\infty )\ni t\longmapsto \frac{\alpha (t)}{t^2}\ln \frac t{\beta (t)}$ is a
positive Lebesgue measurable function.

In particular, if $\alpha : (0\;\! ,+\infty )\longrightarrow (0\;\! ,+\infty )$ is an increasing function and

\noindent $\displaystyle \int\limits_1^{+\infty}\!\frac{\alpha (t)}{t^2}\;\!{\rm d}t<+\infty\;\!$, then
$\displaystyle \int\limits_1^{+\infty}\!\frac{\alpha (t)}{t^2}\ln \frac t{\alpha (t)}\;\!{\rm d}t>-\infty$ is a
well defined improper integral.
Indeed, we have $\displaystyle \lim\limits_{t\to +\infty}\frac{\alpha (t)}t=0$ by Theorem \ref{increasing}.
\smallskip

Let us say that an increasing function $\alpha : (0\;\! ,+\infty )\longrightarrow (0\;\! ,+\infty )$
satisfies the {\it mild strong non-quasianalyticity condition} if
\medskip

\centerline{$\displaystyle \int\limits_1^{+\infty}\!\frac{\alpha (t)}{t^2}\;\!{\rm d}t<+\infty\;$ and
$\displaystyle \int\limits_1^{+\infty}\!\frac{\alpha (t)}{t^2}\ln \frac t{\alpha (t)}\;\!{\rm d}t<+\infty\;\! .$}
\smallskip

\noindent We notice that, for $\omega\in\mathbf{\Omega}\;\!$, $(0\;\! ,+\infty )\ni t\longmapsto
|\omega (t)|\in (0\;\! ,+\infty )$ satisfies the

\noindent mild non-quasianalyticity condition exactly when condition (\ref{w-str}) is satisfied,
that is when $\omega$ satisfies the mild non-quasianalyticity condition as defined in Section 2.
\smallskip

\begin{proposition}\label{msnq-permanence1}
Let $\alpha : (0\;\! ,+\infty )\longrightarrow (0\;\! ,+\infty )$ be an increasing function

\noindent satisfying the mild strong non-quasianalyticity condition. Then
\begin{itemize}
\item[(i)] $c\cdot\alpha$ and $\alpha (L\,\cdot\;\! )$
satisfy the mild strong non-quasianalyticity condition for each $c>0$ and $L>0\;\! ;$
\item[(ii)] $\alpha +\beta$ satisfies the mild strong non-quasianalyticity condition for

\noindent each increasing $\beta : (0\;\! ,+\infty )\longrightarrow (0\;\! ,+\infty )$ satisfying the mild strong
non-quasianalyticity condition$\;\! ;$
\item[(iii)] any increasing $\beta : (0\;\! ,+\infty )\longrightarrow (0\;\! ,+\infty )\;\! ,\beta\leq\alpha\;\! ,$
satisfies the mild strong non-quasianalyticity condition.
\end{itemize}
\end{proposition}

\begin{proof}
The proof of (i) is immediate. Also (ii) is easily seen by using that
\medskip

\centerline{$\displaystyle \frac{\alpha (t)+\beta (t)}{t^2}\ln \frac t{\alpha (t)+\beta (t)}\;\!{\rm d}t\leq
\frac{\alpha (t)}{t^2}\ln \frac t{\alpha (t)}\;\!{\rm d}t+
\frac{\beta (t)}{t^2}\ln \frac t{\beta (t)}\;\!{\rm d}t\;\! .$}
\smallskip

\noindent For the proof of (iii) we notice that, according to Theorem \ref{increasing}, there exists
a $t_0\geq 1$ such that $\displaystyle \frac{\alpha (t)}t<\frac 1e\Leftrightarrow\alpha (t)<\frac te$
for all $t\geq t_0\;\!$. Then
\medskip

\centerline{$\displaystyle 
\beta (t)\ln \frac t{\beta (t)}\leq\alpha (t)\ln \frac t{\alpha (t)}\;\! ,\qquad t\geq t_0\;\! .$}
\smallskip

\noindent Indeed, $\displaystyle  \Big( 0\;\! ,\frac te\Big)\ni x\longmapsto x\ln \frac tx$ is
increasing and $\displaystyle  0<\beta (t)\leq\alpha (t)<\frac te\;\!$.

\end{proof}

At first view, the next characterization of mild strong non-quasianalyticity (more precisely, of
its negation) can appear surprising$\;\! :$

\begin{proposition}\label{msnq-charact}
Let $\alpha : (0\;\! ,+\infty )\longrightarrow (0\;\! ,+\infty )$ be an increasing function

\noindent such that $\displaystyle \int\limits_1^{+\infty}\!\frac{\alpha (t)}{t^2}\;\!{\rm d}t<+\infty\;\!$.
Then the following conditions are equivalent$\, :$
\begin{itemize}
\item[(i)] $\alpha$ does not satisfy the mild strong non-quasianalyticity condition,

\noindent that is $\displaystyle \int\limits_1^{+\infty}\!\frac{\alpha (t)}{t^2}\ln \frac t{\alpha (t)}\;\!{\rm d}t
=+\infty\;\! .$
\item[(ii)] For any increasing function $\beta : (0\;\! ,+\infty )\longrightarrow (0\;\! ,+\infty )$ such that

\noindent $\displaystyle \int\limits_1^{+\infty}\!\frac{\beta (t)}{t^2}\;\!{\rm d}t<+\infty\;\!$, we have
$\displaystyle \int\limits_1^{+\infty}\!\frac{\alpha (t)}{t^2}\ln \frac t{\beta (t)}\;\!{\rm d}t =+\infty\;\! .$
\end{itemize}
The above two conditions imply the condition
\begin{itemize}
\item[(iii)] $\displaystyle \int\limits_e^{+\infty}\!\frac{\alpha (t)}{t^2} \ln\ln (t)\;\!{\rm d}t=+\infty$
\end{itemize}
and, if $\alpha$ is also subadditive or $\alpha =\ln |\omega (\;\!\cdot\;\! )|$ with $\omega\in
\mathbf{\Omega}_0\;\!$, then all the above three conditions are equivalent.
\end{proposition}

\begin{proof}
Implication (ii)$\Rightarrow$(i) is trivial. For (i)$\Rightarrow$(ii) : since
\medskip

\noindent\hspace{0.6 mm}$\displaystyle
\frac{\alpha (t)}{t^2}\ln \frac t{\beta (t)}\geq \frac{\alpha (t)}{t^2}\ln \frac t{\alpha (t)+\beta (t)}
=\frac{\alpha (t)}{t^2}\ln \frac t{\alpha (t)}-\frac{\alpha (t)}{t^2}\ln \frac{\alpha (t)+\beta (t)}{\alpha (t)}$

\noindent\hspace{21.4 mm}$\displaystyle
\geq\frac{\alpha (t)}{t^2}\ln \frac t{\alpha (t)}-\frac{\alpha (t)}{t^2}\frac{\alpha (t)+\beta (t)}{\alpha (t)}
=\frac{\alpha (t)}{t^2}\ln \frac t{\alpha (t)}-\frac{\alpha (t)+\beta (t)}{t^2}\;\! ,$
\smallskip

\noindent we have
\begin{equation*}
\int\limits_1^{+\infty}\!\frac{\alpha (t)}{t^2}\ln \frac t{\beta (t)}\;\!{\rm d}t\geq
\int\limits_1^{+\infty}\!\frac{\alpha (t)}{t^2}\ln \frac t{\alpha (t)}\;\!{\rm d}t -\!
\int\limits_1^{+\infty}\!\frac{\alpha (t)+\beta (t)}{t^2}\;\!{\rm d}t =+\infty\;\! .
\end{equation*}

(ii)$\Rightarrow$(iii) follows by applying (ii) to
\begin{equation*}
\beta (t)=\begin{cases}
\;\displaystyle \frac t{\;\! (\ln t)^2} &\!\!\! \text{for }\, t\geq e^2\, , \\
\;\displaystyle\hspace{4.2 mm} \frac t4 &\!\!\! \text{for }\, 0<t<e^2\, .
\end{cases}
\end{equation*}

Finally we prove that, if $\alpha$ is also subadditive or $\alpha =\ln |\omega (\;\!\cdot\;\! )|$ with

\noindent $\omega\in\mathbf{\Omega}_0\;\!$, then (iii)$\Rightarrow$(i).
For we recall that, according to Theorem \ref{concave},

\noindent $\displaystyle \lim\limits_{t\to +\infty}\frac{\;\! \alpha (t) \ln t}t=0\;\!$. Consequently
there exists some $t_0\geq e$ such that

\noindent $\displaystyle \frac{\;\! \alpha (t) \ln t}t<1\Leftrightarrow\frac t{\alpha (t)}>\ln t$ for $t\geq t_0\;\!$.
We deduce$\;\! :$

\centerline{$\displaystyle \int\limits_{t_0}^{+\infty}\!\frac{\alpha (t)}{t^2}\ln \frac t{\alpha (t)}\;\!{\rm d}t
\geq \int\limits_{t_0}^{+\infty}\!\frac{\alpha (t)}{t^2}\ln \ln (t)\;\!{\rm d}t =+\infty\;\! .$}

\end{proof}

The non-quasianalyticity and mild strong non-quasianalyticity conditions for increasing functions
can be rewritten in discretized form$\;\! :$

\begin{proposition}\label{msnq-discrete}
Let $\alpha : (0\;\! ,+\infty )\longrightarrow (0\;\! ,+\infty )$ be an increasing function.
\begin{itemize}
\item[(i)] $\displaystyle \int\limits_1^{+\infty}\!\frac{\alpha (t)}{t^2}\;\!{\rm d}t<+\infty$ if and only if
$\;\displaystyle \sum\limits_{j=1}^{\infty}\frac{\alpha (2^j)}{2^j}<+\infty\;\!$.
\item[(ii)] Assuming that $\displaystyle \int\limits_1^{+\infty}\!\frac{\alpha (t)}{t^2}\;\!{\rm d}t<+\infty\;\!$,
we have $\displaystyle \int\limits_1^{+\infty}\!\frac{\alpha (t)}{t^2}\ln \frac t{\alpha (t)}\;\!{\rm d}t=+\infty$
if and only if $\;\displaystyle \sum\limits_{j=1}^{\infty}\frac{\alpha (2^j)}{2^j}\ln \frac{2^j}{\beta (2^j)}
=+\infty$ for any increasing function $\beta : (0\;\! ,+\infty )\longrightarrow (0\;\! ,+\infty )$ satisfying
$\displaystyle \int\limits_1^{+\infty}\!\frac{\beta (t)}{t^2}\;\!{\rm d}t<+\infty\;\!$.
\item[(iii)] $\displaystyle \int\limits_e^{+\infty}\!\frac{\alpha (t)}{t^2} \ln\ln (t)\;\!{\rm d}t<+\infty$ if and
only if $\;\displaystyle \sum\limits_{j=1}^{\infty}\frac{\alpha (2^j)}{2^j}\ln j<+\infty\;\!$.
\end{itemize}
\end{proposition}

\begin{proof}
(i) follows by noticing that
\smallskip

\centerline{$\displaystyle
\int\limits_{2^j}^{2^{j+1}}\! \frac{\alpha (t)}{t^2}\;\!{\rm d}t\leq \alpha (2^{j+1})\!
\int\limits_{2^j}^{2^{j+1}}\! \frac 1{t^2}\;\!{\rm d}t =\frac{\;\!\alpha (2^{j+1})}{2^{j+1}}$}
\smallskip

\noindent and
\smallskip

\centerline{$\displaystyle \int\limits_{2^j}^{2^{j+1}}\! \frac{\alpha (t)}{t^2}\;\!{\rm d}t\geq \alpha (2^j)\!
\int\limits_{2^j}^{2^{j+1}}\! \frac 1{t^2}\;\!{\rm d}t =\frac{\;\!\alpha (2^j)}{2^{j+1}}\;\!$.}

Let us now assume that $\displaystyle \int\limits_1^{+\infty}\!\frac{\alpha (t)}{t^2}\;\!{\rm d}t<+\infty$
and $\displaystyle \int\limits_1^{+\infty}\!\frac{\alpha (t)}{t^2}\ln \frac t{\alpha (t)}\;\!{\rm d}t=+\infty\;\!$.
Let further $\beta : (0\;\! ,+\infty )\longrightarrow (0\;\! ,+\infty )$ be an arbitrary increasing function.

\noindent such that $\displaystyle \int\limits_1^{+\infty}\!\frac{\beta (t)}{t^2}\;\!{\rm d}t<+\infty\;\!$.
Then also $\beta (2\,\cdot\;\! )$ is increasing and such that
$\displaystyle \int\limits_1^{+\infty}\!\frac{\beta (2t)}{t^2}\;\!{\rm d}t<+\infty\;\!$,
so Proposition \ref{msnq-charact} yields $\displaystyle  \int\limits_1^{+\infty}\!
\frac{\alpha (t)}{t^2}\ln \frac t{\beta (2t)}\;\!{\rm d}t=+\infty\;\!$.
\smallskip

\noindent Since
\smallskip

\centerline{$\displaystyle \int\limits_{2^j}^{2^{j+1}}\! \frac{\alpha (t)}{t^2}\ln \frac t{\beta (2t)}\;\!{\rm d}t\leq
\alpha (2^{j+1})\ln \frac{2^{j+1}}{\beta (2\cdot 2^j)}\int\limits_{2^j}^{2^{j+1}}\! \frac 1{t^2}\;\!{\rm d}t=
\frac{\alpha (2^{j+1})}{2^{j+1}}\ln \frac{2^{j+1}}{\beta (2^{j+1})}\;\! ,$}
\smallskip

\noindent we obtain
\smallskip

\centerline{$\displaystyle \sum\limits_{j=1}^{\infty}\frac{\alpha (2^j)}{2^j}\ln \frac{2^j}{\beta (2^j)}
\geq  \int\limits_1^{+\infty}\!\frac{\alpha (t)}{t^2}\ln \frac t{\beta (2t)}\;\!{\rm d}t=+\infty\;\! .$}
\smallskip

Assume now that
\smallskip

\centerline{$\displaystyle \int\limits_1^{+\infty}\!\frac{\alpha (t)}{t^2}\;\!{\rm d}t<+\infty$
and $\;\displaystyle \sum\limits_{j=1}^{\infty}\frac{\alpha (2^j)}{2^j}\ln \frac{2^j}{\beta (2^j)}
=+\infty$}
\smallskip

\noindent for any increasing function $\beta : (0\;\! ,+\infty )\longrightarrow (0\;\! ,+\infty )$ with
$\displaystyle \int\limits_1^{+\infty}\!\frac{\beta (t)}{t^2}\;\!{\rm d}t<+\infty\;\!$.
Since $\alpha (2\,\cdot\;\! )$ is such a function, we have
$\;\displaystyle \sum\limits_{j=1}^{\infty}\frac{\alpha (2^j)}{2^j}\ln \frac{2^j}{\alpha (2^{j+1})}=+\infty\;\!$.
Since
\smallskip

\centerline{$\displaystyle
\int\limits_{2^j}^{2^{j+1}}\! \frac{\alpha (t)}{t^2}\ln \frac t{\alpha (t)}\;\!{\rm d}t\geq \alpha (2^j)
\ln \frac{2^j}{\alpha (2^{j+1})}\int\limits_{2^j}^{2^{j+1}}\! \frac 1{t^2}\;\!{\rm d}t =
\frac{\alpha (2^j)}{2^{j+1}}\ln \frac{2^j}{\alpha (2^{j+1})}\;\! ,$}
\smallskip

\noindent we deduce $\;\displaystyle \int\limits_2^{+\infty}\! \frac{\alpha (t)}{t^2}\ln \frac t{\alpha (t)}\;\!{\rm d}t
\geq\frac 12\;\! \sum\limits_{j=1}^{\infty}\frac{\alpha (2^j)}{2^j}\ln \frac{2^j}{\alpha (2^{j+1})}=+\infty\;\!$.
\smallskip

Finally, (iii) follows by using the estimations
\smallskip

\noindent\hspace{3 mm}$\displaystyle
\int\limits_{2^j}^{2^{j+1}}\! \frac{\alpha (t)}{t^2} \ln\ln (t)\;\!{\rm d}t\leq \alpha (2^{j+1})
\ln \ln (2^{j+1})\!\int\limits_{2^j}^{2^{j+1}}\! \frac 1{t^2}\;\!{\rm d}t \leq \frac{\alpha (2^{j+1})}{2^{j+1}}
\ln (j+1)\;\! ,$

\noindent\hspace{3 mm}$\displaystyle
\int\limits_{2^j}^{2^{j+1}}\! \frac{\alpha (t)}{t^2} \ln\ln (t)\;\!{\rm d}t\geq \alpha (2^j)\ln \ln (2^j)\!
\int\limits_{2^j}^{2^{j+1}}\! \frac 1{t^2}\;\!{\rm d}t\geq \frac{\;\!\alpha (2^j)}{2^{j+1}}\;\! \frac{\;\!\ln j}2
=\frac 14\;\! \frac{\alpha (2^j)}{2^j} \ln j\;\! ,$
\smallskip

\noindent the second one, in which $\displaystyle \ln\ln (2^j)\geq \frac{\;\!\ln j}2$ was used,
valid only for $j\geq 3\;\!$.

\end{proof}

The condition for the sequence $\big(\alpha (2^j)\big)_{j\geq 1}\;\!$, formulated in Proposition
\ref{msnq-discrete} to characterize the negation of the mild strong non-quasianalyticity for an increasing
$\alpha : (0\;\! ,+\infty )\longrightarrow (0\;\! ,+\infty )$ satisfying the non-quasianalyticity condition,
has an important permanence property which will be used in the next section to prove Theorem
\ref{no-min}$\;\! :$

\begin{proposition}\label{msnq-permanence2}
Let $\big( a_j\big)_{j\geq 1}$ be a sequence in $[\;\! 0\;\! ,+\infty )$ such that
\begin{equation*}
\sum\limits_{j=1}^{\infty}\frac{a_j}{2^j}<+\infty
\end{equation*}
and
\smallskip

\centerline{$\displaystyle \sum\limits_{j=1}^{\infty}\frac{a_j}{2^j}\ln \frac{2^j}{\beta (2^j)} =+\infty$}

\noindent for any increasing function $\beta : (0\;\! ,+\infty )\longrightarrow (0\;\! ,+\infty )$ with
$\displaystyle \int\limits_1^{+\infty}\!\frac{\beta (t)}{t^2}\;\!{\rm d}t<+\infty\;\!$.
Then the sequence $\big( (a_{j+1}-a_j)^+\big)_{j\geq 1}\;\!$, where
\smallskip

\centerline{$\displaystyle \lambda^+=\begin{cases}
\; \lambda &\!\!\! \text{for }\, \lambda\geq 0\;\! , \\
\; 0 &\!\!\! \text{for }\, \lambda < 0\;\! ,
\end{cases}$}
\smallskip

\noindent has the same two properties.

\end{proposition}

\begin{proof}
First of all,
\smallskip

\centerline{$\displaystyle \sum\limits_{j=1}^{\infty}\frac{\;\! (a_{j+1}-a_j)^+}{2^j}\leq
\sum\limits_{j=1}^{\infty}\frac{\;\! a_{j+1}+a_j}{2^j}\leq 2\sum\limits_{j=1}^{\infty}\frac{a_j}{2^j}<+\infty\;\! .$}
\smallskip

Now let $\beta : (0\;\! ,+\infty )\longrightarrow (0\;\! ,+\infty )$ be any increasing function satisfying

\noindent $\displaystyle \int\limits_1^{+\infty}\!\frac{\beta (t)}{t^2}\;\!{\rm d}t<+\infty\;\!$.
By Theorem \ref{increasing} $\displaystyle \lim\limits_{t\to +\infty}\frac{\beta (t)}t=0\;\!$,
so there is an integer $n_0\geq 1$ such that $\beta (2^n)\leq 2^n$ for $n\geq n_0\;\!$.

For each $n>n_0\;\!$,
\smallskip

\noindent\hspace{6.3 mm}$\displaystyle
\sum\limits_{j=n_0}^{n}\frac{\;\! (a_{j+1}-a_j)^+}{2^j}\ln \frac{2^j}{\beta (2^j)}\geq
\sum\limits_{j=n_0}^{n}\frac{\;\! a_{j+1}-a_j}{2^j}\ln \frac{2^j}{\beta (2^j)}$

\noindent\hspace{1.35 mm}$\displaystyle =
-\;\! \frac{\;\! a_{n_0}}{2^{n_0}} \ln \frac{2^{n_0}}{\beta (2^{n_0})}
+\sum\limits_{j=n_0+1}^n \frac{\;\! a_j}{2^j}\Big( 2 \ln \frac{2^{j-1}}{\beta (2^{j-1})}-\ln \frac{2^j}{\beta (2^j)}\Big)
+\frac{\;\! a_{n+1}}{2^n} \ln \frac{2^n}{\beta (2^n)}$

\noindent\hspace{1.35 mm}$\displaystyle \geq
-\;\! \frac{\;\! a_{n_0}}{2^{n_0}} \ln \frac{2^{n_0}}{\beta (2^{n_0})}
+\sum\limits_{j=n_0+1}^n \frac{\;\! a_j}{2^j}\ln \Big(\frac{2^{2j-2}}{\beta (2^{j-1})^2}\cdot \frac{\beta (2^j)}{2^j}\Big)$

\noindent\hspace{1.35 mm}$\displaystyle \geq
-\;\! \frac{\;\! a_{n_0}}{2^{n_0}} \ln \frac{2^{n_0}}{\beta (2^{n_0})}
+\sum\limits_{j=n_0+1}^n \frac{\;\! a_j}{2^j}\ln \frac{2^{j-2}}{\beta (2^j)}\;\! ,$

\noindent so
\medskip

\noindent\hspace{1.5 cm}$\displaystyle
\sum\limits_{j=n_0}^{n}\frac{\;\! (a_{j+1}-a_j)^+}{2^j}\ln \frac{2^j}{\beta (2^j)}$

\noindent\hspace{1 cm}$\displaystyle \geq\,
-\;\! \frac{\;\! a_{n_0}}{2^{n_0}} \ln \frac{2^{n_0}}{\beta (2^{n_0})}
-\! \sum\limits_{j=n_0+1}^n \frac{\;\! a_j}{2^j} \ln 4
+\! \sum\limits_{j=n_0+1}^n \frac{\;\! a_j}{2^j} \ln \frac{2^j}{\beta (2^j)}=+\infty\;\! .$

\end{proof}

We denote

\centerline{$\displaystyle \ln^+t:=\begin{cases}
\; \ln t &\!\!\! \text{for }\, t>0\ \\
\;\hspace{1.8 mm} 0 &\!\!\! \text{for }\, t\leq 0
\end{cases}\; .$}
\medskip

\noindent The next lemma completes Proposition \ref{criteria1}$\;\! :$

\begin{lemma}\label{criteria2}
For $\displaystyle 0<t_1\leq t_2\leq t_3\leq\;\! ...\;\! \leq +\infty\, ,\, t_1<+\infty\;\!$,
are equivalent$\;\! :$
\begin{itemize}
\item[(i)] $\,\displaystyle \sum\limits_{j=2}^{\infty}\frac{\ln\ln j}{t_j}<+\infty\;\!$;
\item[(ii)] $\,\displaystyle \lim\limits_{j\to\infty}t_j=+\infty$ and
$\,\displaystyle \sum\limits_{j=1}^{\infty}\frac{\;\!\ln^+\ln t_j}{t_j}<+\infty\;\!$.
\end{itemize}
Moreover, $({\rm i})$ and $({\rm ii})$ imply
\begin{itemize}
\item[(iii)] $\,\displaystyle \lim\limits_{j\to\infty}\frac{t_j}j=+\infty$ and $\,\displaystyle
\sum\limits_{j=1}^{\infty}\frac{\;\!\displaystyle \ln^+ \frac{t_j}j}{t_j}<+\infty\;\!$,
\end{itemize}
and, if $\displaystyle 0<t_1\leq \frac{t_2}2\leq \frac{t_3}3\leq\;\! ...\;\!$, then $({\rm i})$, $({\rm ii})$
and $({\rm iii})$ are all equivalent.
\end{lemma}

\begin{proof}
First we prove that (i) implies (ii) and (iii).
\smallskip

Clearly, $t_j\longrightarrow +\infty\;\!$. For each $j\geq 3\;\!$, if $t_j\leq j^2$ then
\begin{equation*}
\frac{\ln^+\ln t_j}{t_j}\leq\frac{\ln\ln (j^2)}{t_j}=2\;\! \frac{\ln\ln j}{t_j}\;\! ,
\end{equation*}
while if $t_j>j\;\!$, then
\begin{equation*}
\frac{\ln^+\ln t_j}{t_j}=\frac{\ln\ln t_j}{t_j}<\frac{\ln\ln (j^2)}{j^2}=2\;\! \frac{\ln\ln j}{j^2}\;\! .
\end{equation*}
Therefore
\smallskip

\centerline{$\displaystyle
\sum\limits_{j=3}^{\infty}\frac{\ln^+\ln t_j}{t_j}\leq 2\;\! \sum\limits_{j=3}^{\infty}\;\!\frac{\ln\ln j}{t_j}+
2\;\! \sum\limits_{j=3}^{\infty}\;\! \frac{\ln\ln j}{j^2}
<+\infty\;\! .$}
\smallskip

On the other hand, since $\ln\ln j>1$ for all $j\geq 16\;\!$, we have
$\displaystyle \sum\limits_{j=1}^{\infty}\;\! \frac 1{t_j}<+\infty\;\!$.
Furthermore, for each $j\geq 6\;\!$, if $\displaystyle \frac{t_j}j\leq (\ln j)^2$ then
\smallskip

\centerline{$\displaystyle \frac{\;\!\displaystyle \ln^+ \frac{t_j}j}{t_j}\leq\frac{\;\! \ln (\ln j)^2}{t_j} =2\;\!
\frac{\;\!\ln\ln j}{t_j}\;\! ,$}
\smallskip

\noindent while if $\displaystyle \frac{t_j}j> (\ln j)^2>e\;\!$, then
\medskip

\centerline{$\displaystyle \frac{\;\!\displaystyle \ln^+ \frac{t_j}j}{\displaystyle \frac{t_j}j}=
\frac{\;\!\displaystyle \ln \frac{t_j}j}{\displaystyle \frac{t_j}j}<
\frac{\;\! \ln (\ln j)^2}{(\ln j)^2}=2\;\! \frac{\;\! \ln\ln j}{\;\! (\ln j)^2}\;\! ,\;\!\text{hence }
\frac{\;\!\displaystyle \ln^+ \frac{t_j}j}{t_j}<2\;\! \frac{\;\! \ln\ln j}{\;\! j\;\! (\ln j)^2}\;\! .$}
\medskip

\noindent We conclude that
\smallskip

\centerline{$\displaystyle
\sum\limits_{j=6}^{\infty}\frac{\;\!\displaystyle \ln^+ \frac{t_j}j}{\displaystyle \frac{t_j}j}\leq
2\;\! \sum\limits_{j=6}^{\infty}\frac{\;\!\ln\ln j}{t_j} +2\;\! \sum\limits_{j=6}^{\infty}
\frac{\;\! \ln\ln j}{\;\! j\;\! (\ln j)^2}<+\infty\;\! .$}
\smallskip

Next we prove implication (ii)$\,\Rightarrow$(i).

Since $t_j\longrightarrow +\infty\;\!$, we have eventually $\ln^+\ln t_j\geq 1\;\!$, hence
$\displaystyle \sum\limits_{j=1}^{\infty}\;\! \frac 1{t_j}<+\infty\;\!$.
Consequently, (see e.g. \cite{C-Z1}, Lemma 1.5 (ii)), $\displaystyle \lim\limits_{j\to\infty}
\frac{t_j}j=+\infty\;\!$. In particular, we have eventually $j\leq t_j$ and the convergence of
$\,\displaystyle \sum\limits_{j=2}^{\infty}\frac{\ln\ln j}{t_j}$ follows.

Finally we show that if $\displaystyle 0<t_1\leq \frac{t_2}2\leq \frac{t_3}3\leq\;\! ...\;\!$, then
(iii)$\,\Rightarrow$(i).

Since $\displaystyle \frac{t_j}j\longrightarrow +\infty\;\!$, we have eventually
$\displaystyle \ln^+ \frac{t_j}j\geq 1\;\!$, and thus $\displaystyle \sum\limits_{j=1}^{\infty}\;\! \frac 1{t_j}
<+\infty\;\!$. By \cite{C-Z1}, Lemma 1.5 (iii) it follows that $\displaystyle \frac{t_j}{\;\! j\ln j}\longrightarrow
+\infty\;\!$. In particular, we

\noindent have eventually $\displaystyle \ln j\leq\frac{t_j}j$ and the convergence of
$\,\displaystyle \sum\limits_{j=2}^{\infty}\frac{\ln\ln j}{t_j}$ follows.

\end{proof}

We end this section with a summary of several characterizations the mild strong non-quasianalyticity
condition for functions of the form $|\;\!\omega (\;\!\cdot\;\! )|$ with $\omega\in\mathbf{\Omega}\;\!$.

\begin{theorem}\label{msnq-omega}
For $\;\!\displaystyle 0<t_1\leq t_2\leq t_3\leq\;\! ...\;\! \leq +\infty\, ,\, t_1<+\infty\, ,\,
\sum\limits_{j=1}^{\infty}\frac 1{t_j}<+\infty\;\!$, let us denote
\begin{equation*}
\begin{split}
n(t)=\;&\#\{ k\geq 1 ; t_k\leq t\}\;\! ,\qquad t>0\;\! , \\
N(t)=\;&\ln \max \Big( 1\;\! ,\sup\limits_{k\geq 1}\frac{t^k}{t_1\;\! t_2\;\! ...\;\! t_k}\Big)\;\! ,\qquad t>0\;\! , \\
\omega (z)=\;&\prod\limits_{j=1}^{\infty}\Big( 1+\frac{iz}{t_j}\Big)\;\! ,z\in\mathbb{C}\;\! .
\end{split}
\end{equation*}

Then the following conditions are equivalent$\;\! :$
\begin{itemize}
\item[(i)] $\;\;\displaystyle \sum\limits_{j=1}^{\infty}\frac{\;\!\displaystyle \ln \frac{t_j}j}{t_j}<+\infty\;\! ;$
\item[(ii)] $\;\displaystyle \int\limits_1^{+\infty}\frac{n(t)|}{t^2} \ln \frac t{n(t)}
\;\!{\rm d}t<+\infty\;\!$;
\item[(iii)] $\;\displaystyle \int\limits_1^{+\infty}\frac{N(t)|}{t^2} \ln \frac t{N(t)}
\;\!{\rm d}t<+\infty\;\!$;
\item[(iv)] $\;\displaystyle \int\limits_1^{+\infty}\frac{\ln |\omega (t)|}{t^2} \ln \frac t{\ln |\omega (t)|}
\;\!{\rm d}t<+\infty\;\!$.
\end{itemize}
$($In the above conditions we take $\displaystyle 0\ln \frac 10=0$ when it occurs.$)$

The above conditions are implied by the next equivalent conditions$\;\! :$
\begin{itemize}
\item[(v)] $\;\;\displaystyle \sum\limits_{j=2}^{\infty}\frac{\;\!\ln\ln j}{t_j}<+\infty\;\! ;$
\item[(vi)] $\;\;\displaystyle \sum\limits_{j=2}^{\infty}\frac{\;\!\ln^+\ln t_j}{t_j}<+\infty\;\! .$
\end{itemize}

Finally, if $\displaystyle 0<t_1\leq \frac{t_2}2\leq \frac{t_3}3\leq\;\! ...\;\!$, then all the above
six conditions are

\noindent equivalent.
\end{theorem}

\begin{proof}
Statement (i)$\Leftrightarrow$(ii)$\Leftrightarrow$(iii)$\Leftrightarrow$(iv) is
\cite{C-Z2}, Lemma 2.1, while (v)$\Leftrightarrow$(vi) and the relationship between
the above two groups of equivalent conditions is Lemma \ref{criteria2}.

\end{proof}

\section{Proof of the negative minimum modulus theorem}

In this section we provide a proof for Theorem \ref{no-min}, a negative minimum modulus theorem.
The idea of the proof, located in the proof of the next Lemma \ref{main-lemma}, is due to W. K. Hayman
(\cite{H}), while the technical execution is based upon the topics of Section 5.

\begin{lemma}\label{main-lemma}
Let $n_1\;\! ,n_2\;\!,\;\! ...\;\!\geq 0$ be integers such that
\begin{equation}\label{main-lemma-cond1}
\sum\limits_{j=1}^{\infty}\frac{n_j}{2^j}<+\infty
\end{equation}
and
\begin{equation}\label{main-lemma-cond2}
\sum\limits_{j=1}^{\infty}\frac{n_j}{2^j} \ln\frac{2^j}{\;\!\beta (2^j)}=+\infty
\end{equation}
for any increasing $\beta : (0\;\! ,+\infty )\longrightarrow (0\;\! ,+\infty )$ such that
$\displaystyle \int\limits_1^{+\infty}\!\frac{\beta (t)}{t^2}\;\!{\rm d}t<+\infty\;\!$.
The the formulas
\begin{equation*}
\omega_0(z)=\prod\limits_{j=1}^{\infty}\big( 1+\frac{iz}{2^j}\Big)^{\! n_j}\;\! ,\;
f(z)=\prod\limits_{j=1}^{\infty}\bigg( 1-\Big(\frac z{2^j}\Big)^{\! 2}\bigg)^{\! n_j},\qquad z\in\mathbb{C}
\end{equation*}
define a function $\omega_0\in\mathbf{\Omega}$ and an entire function $f$ with
\begin{equation}\label{main-lemma1}
|f(z)|\leq \big|\omega_0(|z|)\big|^2\;\! ,\qquad z\in\mathbb{C}\;\! ,
\end{equation}
such that there exists no increasing function $\beta : (0\;\! ,+\infty )\longrightarrow (0\;\! ,+\infty )$ with $\displaystyle \int\limits_1^{+\infty}\! \frac{\beta (t)}{t^2}\;\!{\rm d}t<+\infty$ satisfying
\begin{equation}\label{main-lemma2}
\sup\limits_{\substack{s\in\mathbb{R} \\ |s-t|\leq \beta (t)}} \ln |f(s)|\geq -\;\! \beta (t)\;\! ,\qquad t>0\;\! .
\end{equation}
\end{lemma}

\begin{proof}
(\ref{main-lemma-cond1}) yields $\omega_0\in\mathbf{\Omega}$ and,
since $\displaystyle \Big| 1-\Big(\frac z{2^j}\Big)^{\! 2}\Big|\leq 1+\Big(\frac{|z|}{2^j}\Big)^{\! 2}=
\Big| 1+\frac{\;\! i\;\! |z|}{2^j}\Big|^2\;\!$, (\ref{main-lemma1}) holds true.

For the remaining part of the proof, we need a particular upper estimate of $|f(z)|$ for $z$ in the disk
of radius $2^j$, centered at $2^j$.

Let $j\geq 1$ be arbitrary. Since $2^j$ is a zero of multiplicity $n_j$ of $f$, we can apply the general Schwarz'
lemma (see e.g. \cite{Pr}, Chapter XII, \S 3, Section 2, page 359, or \cite{Re}, Chapter 9,
\S 2, Exercise 1, page 274), obtaining for any $z\in\mathbb{C}\;\! , |z-2^j|\leq 2^j$,
\begin{equation*}
\begin{split}
|f(z)|\leq\;& \Big( \sup\limits_{|z'-2^j|=2^j}|f(z')|\Big)\! \Big(\frac{|z-2^j|}{2^j}\Big)^{\! n_j} \\
\overset{\rm (\ref{main-lemma1})}{\leq}&
\Big( \sup\limits_{|z'-2^j|=2^j}\big|\omega_0(|z'|)\big|^2\Big)\! \Big(\frac{|z-2^j|}{2^j}\Big)^{\! n_j}
=\big|\omega_0(2^{j+1})\big|^2\Big(\frac{|z-2^j|}{2^j}\Big)^{\! n_j}\;\! .
\end{split}
\end{equation*}
Thus, for each $0<\delta\leq 1\;\!$,
\begin{equation}\label{main-lemma3}
\ln |f(z)|\leq 2 \ln \big|\omega_0(2^{j+1})\big| +n_j\ln\delta\;\! ,\qquad z\in\mathbb{C}\;\! , |z-2^j|\leq 2^j\delta\;\! .
\end{equation}

Now we assume that for some increasing $\beta : (0\;\! ,+\infty )\longrightarrow (0\;\! ,+\infty )$ with

\noindent $\displaystyle \int\limits_1^{+\infty}\! \frac{\beta (t)}{t^2}\;\!{\rm d}t<+\infty$ we have
(\ref{main-lemma2}) and show that this assumption leads to a

\noindent contradiction.

By (\ref{main-lemma2}) we have
\begin{equation}\label{main-lemma4}
\sup\limits_{\substack{s\in\mathbb{R} \\ |s-2^j|\leq \beta (2^j)}} \ln |f(s)|\geq -\;\! \beta (2^j)\;\! ,\qquad j\geq 1\;\! .
\end{equation}
On the other hand, using Theorem \ref{increasing}, we deduce $\displaystyle \lim\limits_{t\to\infty}
\frac{\beta (t)}t=0\;\!$, so there exists $j_0\geq 1$ such that
\begin{equation*}
\frac{\beta (2^j)}{2^j}\leq 1\;\! ,\qquad j\geq j_0\;\! .
\end{equation*}
Applying (\ref{main-lemma3}) with $\displaystyle \delta=\frac{\beta (2^j)}{2^j}\;\!$, we obtain
\begin{equation}\label{main-lemma5}
\sup\limits_{\substack{z\in\mathbb{C} \\ |z-2^j|\leq \beta (2^j)}}\ln |f(z)|\leq 2 \ln \big|\omega_0(2^{j+1})\big|
+n_j\ln \frac{\beta (2^j)}{2^j}\;\! ,\qquad j\geq j_0\;\! .
\end{equation}

(\ref{main-lemma4}) and (\ref{main-lemma5}) imply successively for every $j\geq j_0$
\begin{equation*}
\begin{split}
&-\;\! \beta (2^j)\leq 2 \ln \big|\omega_0(2^{j+1})\big| +n_j\ln \frac{\beta (2^j)}{2^j}\;\! , \\
&\hspace{3 mm}n_j\ln \frac{2^j}{\beta (2^j)}\leq 2 \ln \big|\omega_0(2^{j+1})\big| +\beta (2^j)\;\! .
\end{split}
\end{equation*}
Consequently
\begin{equation*}
\sum\limits_{j=j_0}^{\infty}\frac{n_j}{2^j}\ln \frac{2^j}{\beta (2^j)}\leq 4\sum\limits_{j=j_0}^{\infty}
\frac{\ln \big|\omega_0(2^{j+1})\big|}{2^{j+1}} +\sum\limits_{j=j_0}^{\infty}\frac{\beta (2^j)}{2^j}\;\! .
\end{equation*}
But this is not possible, because the left-hand side of the above inequality is $+\infty$
according to the assumption (\ref{main-lemma-cond2}), while the right-hand side is finite
because of Remark \ref{explicite} (4) and Proposition \ref{msnq-discrete} (i).

\end{proof}

Now we are ready to prove Theorem \ref{no-min}$\;\! ,$ the main goal of this section :
\medskip

%
%


{\it Proof} (of {\bf Theorem \ref{no-min}}$\;\!\!$).
Let $\;\!\displaystyle 0<t_1\leq t_2\leq t_3\leq\;\! ...\;\! \leq +\infty\, ,\, t_1<+\infty\, ,\,
\sum\limits_{j=1}^{\infty}\frac 1{t_j}<+\infty\;\!$, be such that
\smallskip

\centerline{$\displaystyle \omega (z)=\prod\limits_{j=1}^{\infty}\Big( 1+\frac{iz}{t_j}\Big)\;\! ,
\qquad z\in\mathbb{C}\;\! .$}
\smallskip

Let us denote, for every $t>0\;\!$, by $n(t)$ the number of the elements of

\noindent the set $\{ k\geq 1 ; t_k\leq t\}\;\!$. By Remark \ref{explicite} (1) we have
$\displaystyle \int\limits_{1}^{\infty}\frac{\;\! n(t)}{t^2}\;\!{\rm d}t<+\infty$ and,

\noindent according to Proposition \ref{msnq-discrete} (i), it follows
$\displaystyle \sum\limits_{j=1}^{\infty}\frac{\;\! n(2^j)}{2^j}<+\infty\;\!$.

On the other hand, Theorem \ref{msnq-omega} yields
\begin{equation*}
\int\limits_1^{+\infty}\frac{n(t)|}{t^2} \ln \frac t{n(t)}\;\!{\rm d}t=+\infty\;\! .
\end{equation*}
Applying Proposition \ref{msnq-discrete} (ii) to $(0\;\! ,+\infty )\ni t\longmapsto n(t)\;\!$, we deduce that
\medskip

\centerline{$\displaystyle \sum\limits_{j=1}^{\infty}\frac{n(2^j)}{2^j}\ln \frac{2^j}{\beta (2^j)}=+\infty$}

\noindent for any increasing $\beta : (0\;\! ,+\infty )\longrightarrow (0\;\! ,+\infty )$ satisfying
$\displaystyle \int\limits_1^{+\infty}\!\frac{\beta (t)}{t^2}\;\!{\rm d}t<+\infty\;\!$.

Set
\begin{equation*}
n_1:=n(2)\;\! ,\qquad n_j:=n(2^j)-n(2^{j-1})\text{ for }j\geq 2\;\! .
\end{equation*}
By Proposition \ref{msnq-permanence2} we infer that $\displaystyle \sum\limits_{j=1}^{\infty}
\frac{\;\! n_j}{2^j}<+\infty$ and

\centerline{$\displaystyle \sum\limits_{j=1}^{\infty}\frac{n_j}{2^j}\ln \frac{2^j}{\beta (2^j)}=+\infty$}

\noindent for any increasing $\beta : (0\;\! ,+\infty )\longrightarrow (0\;\! ,+\infty )$ satisfying
$\displaystyle \int\limits_1^{+\infty}\!\frac{\beta (t)}{t^2}\;\!{\rm d}t<+\infty\;\!$.
In other words, the sequence $(n_j)_{j\geq 1}$ satisfies conditions (\ref{main-lemma-cond1})
and (\ref{main-lemma-cond2}), so Lemma \ref{main-lemma} implies that the formula
\smallskip

\centerline{$\displaystyle
f(z)=\prod\limits_{j=1}^{\infty}\bigg( 1-\Big(\frac z{2^j}\Big)^{\! 2}\bigg)^{\! n_j},\qquad z\in\mathbb{C}$}
\smallskip

\noindent defines an entire function $f$ such that there exists no increasing function
$\beta : (0\;\! ,+\infty )\longrightarrow (0\;\! ,+\infty )$ with
$\displaystyle \int\limits_1^{+\infty}\! \frac{\beta (t)}{t^2}\;\!{\rm d}t<+\infty$ satisfying
(\ref{main-lemma2}) = (\ref{minimum}).

It remains only to verify (\ref{dominated})$\;\! :$ we have for every $z\in\mathbb{C}$
\medskip

\noindent\hspace{1.11 cm}$\displaystyle
|f(z)|\leq\prod\limits_{j=1}^{\infty}\bigg( 1+\Big(\frac{|z|}{2^j}\Big)^{\! 2}\bigg)^{\! n_j}$

\noindent\hspace{2.16 cm}$\displaystyle =
\bigg( 1+\Big(\frac{|z|}{2}\Big)^{\! 2}\bigg)^{\! n(2)}
\prod\limits_{j=2}^{\infty}\bigg( 1+\Big(\frac{|z|}{2^j}\Big)^{\! 2}\bigg)^{\! n(2^j)-n(2^{j-1})}$

\noindent\hspace{2.16 cm}$\displaystyle \leq
\bigg[ \prod\limits_{k=1}^{n(2)}\bigg( 1+\Big(\frac{|z|}{t_k}\Big)^{\! 2}\bigg)^{\! n_j}\bigg]
\prod\limits_{j=2}^{\infty} \bigg[  \prod\limits_{k=n(2^{j-1})+1}^{n(2^j)}
\bigg( 1+\Big(\frac{|z|}{t_k}\Big)^{\! 2}\bigg)\bigg]$

\noindent\hspace{2.16 cm}$\displaystyle =
\prod\limits_{k=1}^{\infty}\bigg( 1+\Big(\frac{|z|}{t_k}\Big)^{\! 2}\bigg)
=\big|\omega (|z|)\big|^2\;\! .$

\noindent\hspace{12.3 cm}$\square$

\end{document}